\documentclass[a4paper,twoside,openright,titlepage]{article}
\usepackage{a4}
\usepackage[latin1]{inputenc}
\usepackage[english]{babel}
\usepackage{amsmath}
\usepackage{amssymb}
\usepackage{latexsym}
\usepackage{epsfig}
\usepackage{subfigure}

\usepackage{html}
\bodytext{bgcolor="#ffffff"}

\newtheorem{defn}{Definition}[section]
\newtheorem{example}[defn]{Example}
\newtheorem{corollary}[defn]{Corollary}
\newtheorem{lemma}[defn]{Lemma}
\newtheorem{theorem}[defn]{Theorem}

\newenvironment{proof}[1][Proof]{\begin{trivlist}
\item[\hskip \labelsep {\bfseries #1}]}{\end{trivlist}}

\newenvironment{remark}[1][Remark]{\begin{trivlist}
\item[\hskip \labelsep {\bfseries #1}]}{\end{trivlist}}

\newcommand{\coord}{$(x^1, \ldots ,x^n)$}
\newcommand{\man}[1]{$\mathcal{#1}$}
\newcommand{\real}[1]{$\mathbb{R}^{#1}$}

\newcommand{\tpm}{$\mathcal{T}_p\mathcal{M}$}

\title{\huge{The Schwarzian Curvature}}
\author{\Large{Kambiz Fathi} \\[14cm] Department of Mathematics, Uppsala University, Uppsala, Sweden \\ \ \\UUMD Preprint 1998:P8 \\ \date{1998}}

\begin{document}
\selectlanguage{english}

\maketitle

\clearpage{\pagestyle{empty}\cleardoublepage}

\begin{abstract}
We start with introducing one of the most fundamental notions of differential geometry, Manifolds. We present some properties and constructions such as submanifolds, tangent spaces and the tangent map. Then we continue with introducing the real and complex projective space, and describe them from some different points of view. This part is finished by showing that $\mathbb{CP}^{n}$ is a Grassmannian manifold. At this stage we are ready to present the main subject of this thesis. \\
The Schwarzian curvature, usually seems to be an accidental by--product of the calculations, can be seen as a geometric interpretation of the Schwarzian derivative. Harley Flanders  \cite{Fl70} interpreted the Schwarzian derivative of a $C'''$ function as a curvature for curves in the projective line by using the moving frame method of \'{E}lie Cartan. The same argumentation was extended by Weiqi Gao \cite{G94} to obtain the Schwarzian curvatures for curves in higher dimensional projective spaces. \\
I have aimed to give a detailed presentation of Gao's work, where he presented the general formulas for the Schwarzian curvatures for curves in $\mathbb{CP}^{n}$ and gives some properties for the behaviour of the formulas, for example the transformation rules under change of coordinates. The Schwarzian curvatures for curves in $\mathbb{CP}$, $\mathbb{CP}^{2}$ and $\mathbb{CP}^{3}$ are calculated, and some examples are given.
\end{abstract}

\clearpage{\pagestyle{empty}\cleardoublepage}

\setcounter{page}{1}

\tableofcontents

\clearpage{\pagestyle{empty}\cleardoublepage}

\section{Introduction}
The notion of curvature is of great importance in the study of curves and various types of curvature are introduced in differential geometry. In this paper we concern about a specific type of curvature defined for analytic curves on a special type of manifold, called the projective space. For this purpose we construct a moving frame on curves in the complex projective space $\mathbb{CP}^{n}$ in terms of their liftings to the complex space $\mathbb{C}^{n+1} \setminus \{ 0 \}$. Then we define the Schwarzian curvatures, which can be seen as a geometric interpretation of the Schwarzian derivative, for curves in $\mathbb{CP}^{n}$ in terms of the normalized lifting and its derivatives.
\[
f^{(n+1)}= \kappa_{0} f + \ldots + \kappa_{n-1} f^{(n-1)}.
\]
Where the $\kappa$'s denote the Schwarzian curvatures and $f$ denotes the normalized lifting from $\mathbb{CP}^{n}$ to $\mathbb{C}^{n+1} \setminus \{ 0 \}$.

We continue with proving the invariance of the Schwarzian curvatures under affine non--singular transformations and give a slight introduction to the relation between the $\kappa$'s and the structure of the curve $\Phi$ in $\mathbb{CP}^{n}$. We present the general formula for calculation of the Schwarzian curvatures for curves in $\mathbb{CP}^{n}$ and calculate those for cases of $n=1, 2, 3$. Finally we present the transformation rules for in $\mathbb{CP}$ and $\mathbb{CP}^{2}$ under change of coordinates in the domain of the curve.

Almost all of the work about the Schwarzian curvature is a more detailed presentation of Weiqi Gao's paper \cite{G94}, expanded by the calculation of the $\kappa$'s for curves in $\mathbb{CP}^{3}$.

\clearpage{\pagestyle{empty}\cleardoublepage}

\section{Manifolds}
Manifolds are generalizations of our intuitive ideas about curves and surfaces to arbitrary dimensional objects. A curve in the three--dimensional Euclidean space is parameterized locally by a single number $t$ as $(x(t),y(t),z(t))$, while two numbers parameterize a surface as $(x(u,v),y(u,v),z(u,v))$. A curve and a surface are considered locally homeomorphic to $\mathbb{R}$ and \real{2}, respectively. A manifold denoted by \man{M} is a topological space which is homeomorphic to \real{n} locally, it may be different from \real{n} globally. The local homeomorphism enables us to give a point in a manifold a set of $n$ numbers called local coordinates. If a manifold is not homeomorphic to \real{n} globally, we have to introduce several local coordinates. What we then require is that the transition from one coordinate to another is smooth, i.e. is of class $C^{\infty}$.\\

\begin{defn}
\label{sec:atlas}
Let $\mathcal{M}$ be a Hausdorff topological space. A family $\mathcal{U} = \{ (U_i , \varphi_{i})\}_{i \in I}$ is called a $C^{\infty}$--atlas of dimension $n$ on $\mathcal{M}$ if:
\begin{enumerate}
\item $\{U_{i}\}$ is an open covering of $\mathcal{M}$, i.e. the sets $U_{i}$ are open and $\mathcal{M} = \bigcup_{i\in I} U_{i}$ ;
\item each $\varphi_{i} : U_{i} \longrightarrow \mathbb{R}^{n}$ is a homeomorphism onto an open set $U'_{i} \subset \mathbb{R}^{n}$;
\item for any $i,j\in I$ such that $U_{i}\cap U_{j} \neq \emptyset$, the map
\[ \Psi_{ij} = \varphi_i \circ \varphi_{j}^{-1} : \varphi_{j}(U_{i} \cap U_{j}) \longrightarrow \varphi_{i}(U_{i} \cap U_{j})\]
is infinitely differentiable.
\end{enumerate}
The mappings $\Psi_{ij}$ are called coordinate transformations.
\end{defn}
Note that if \man{U} and $\mathcal{U}'$ are $C^{\infty}$--atlases on \man{M} then also $\mathcal{U} \cup \mathcal{U}'$ is a $C^{\infty}$--atlas on \man{M}.

\begin{defn}
Let \man{U} be an atlas on \man{M}. The completion of \man{U} is the family of pairs $( V,\psi )$ with the following properties.
\begin{enumerate}
\item $V$ is an open subset of \man{M};
\item $\psi : V\longrightarrow \mathbb{R}^{n}$ is a homeomorphism onto an open set $V' \subset \mathbb{R}^{n}$;
\item for any $( U,\varphi )\in \mathcal{U}$ such that $U\cap V \neq \emptyset$, the maps
\[ \varphi \circ \psi^{-1} : \psi (U\cap V)\longrightarrow \varphi (U\cap V) \]
and 
\[ \psi \circ \varphi^{-1} : \varphi (U\cap V)\longrightarrow \psi (U\cap V) \]
are infinitely differentiable.
\end{enumerate}
\end{defn}
It is not difficult to check that the completion of \man{U} is also an atlas on \man{M} and that this atlas is maximal in the sense of inclusion.

\begin{defn}
A Hausdorff topological space \man{M} with a $C^{\infty}$--atlas \man{U} is called a differentiable manifold. Elements of \man{U} are called charts on \man{M}. If $(U,\varphi)$ is a chart, then $U$ is called the chart domain and $\varphi$ a local coordinate system (in $U$).
\end{defn} 

Some variants of differentiable manifolds are $C^k$--manifolds, analytic manifolds and complex manifolds. $C^k$--manifolds and analytic manifolds are simply defined by replacing $C^\infty$--differentiability of the coordinate transformation by that of $C^k$--differentiability or $C^\omega$--differentiability, in the definition \ref{sec:atlas}. Even the definition of complex manifold is the same as that of differentiable manifold, except that the local homeomorphisms are required to map from open subsets of the space $\mathbb{C}^n$ (instead of $\mathbb{R}^n$), and the change of coordinates $\Psi_{ij}$ are required to be holomorphic instead of $C^\infty$--differentiable.\\
Since $\mathbb{C}^n =\mathbb{R}^{2n}$ we may regard a complex manifold as a real differentiable manifold whose real dimension is twice the complex dimension. Moreover a complex manifold carries a real analytic structure.\\
We do not require that $\mathcal{M}$ is $\mathbb{R}^n$ globally, but from definition \ref{sec:atlas} we see that $\mathcal{M}$ is locally carrying the Euclidean structure, i.e. in each coordinate neighborhood $U_i$, $\mathcal{M}$ looks like a subset of $\mathbb{R}^n$.

\begin{remark}
A point $p\in\mathcal{M}$ exists independently of its coordinates, thus the choice of coordinates is free. We denote the coordinates of a point $p$ by $( x^1(p),\ldots ,x^n(p))$, where $\varphi = (x^{1}, \ldots ,x^{n})$.
\end{remark}

\begin{example}
The Euclidean space $\mathbb{R}^n$ is of course a differentiable manifold.\\
Let $Id :\mathbb{R}^n \longrightarrow \mathbb{R}^n$ be the identity map. Then the $(\mathbb{R}^n,Id)$ constitutes an atlas for $\mathbb{R}^n$ all by itself.
\end{example}

\begin{remark}
The example above formalizes the fact that the notion of manifold is a generalization of the Euclidean space.
\end{remark}

\begin{example}
The circle $S^1$ is a one--dimensional real manifold. Take the unit circle written in $\mathbb{R}^2$ as $x^2 +y^2 = 1$, then we can choose the charts as follows:
\[
\left\{ \begin{array}{ll}
U_1 & =\{(x,y)\in S^1 : x>0\}\\
U_2 & =\{(x,y)\in S^1 : x<0\}\\
U_3 & =\{(x,y)\in S^1 : y>0\}\\
U_4 & =\{(x,y)\in S^1 : y<0\}
\end{array}
\right.
\]
\begin{figure}[H]
\centering
\epsfig{file=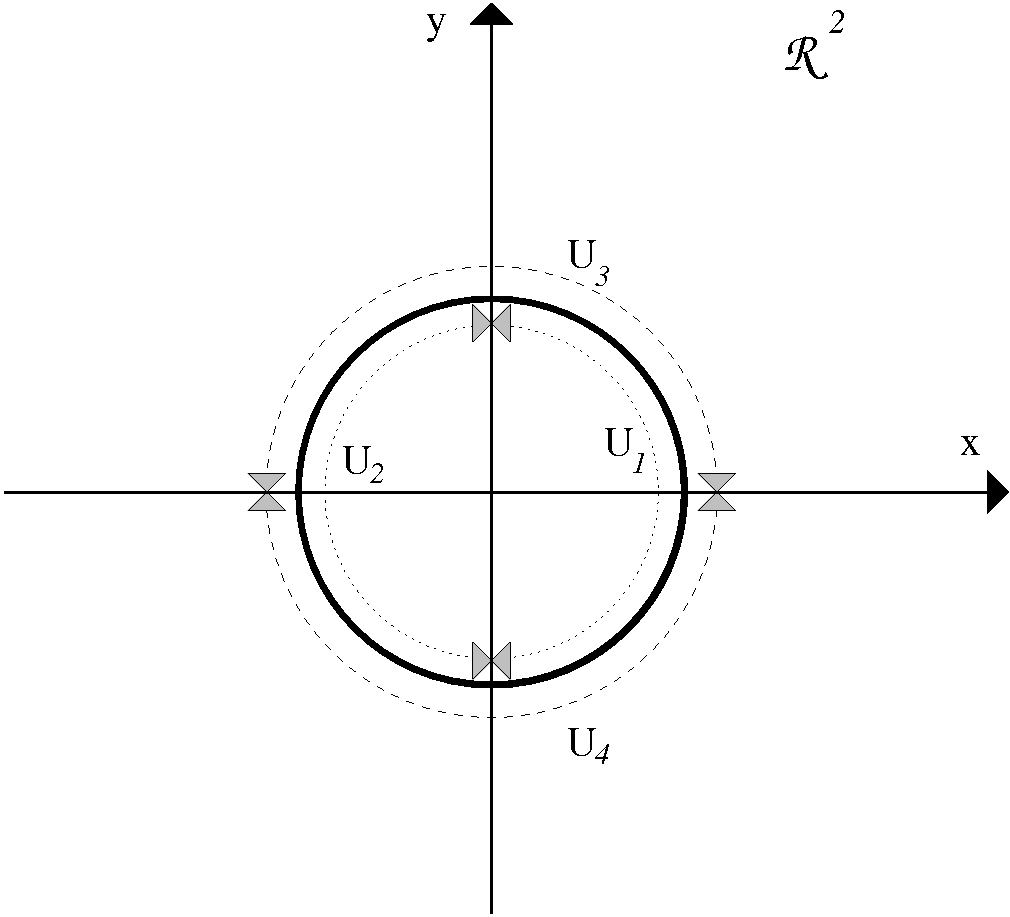, height=6cm}
\caption{The charts of the unit circle in $\mathbb{R}^2$.}
\label{fig:s1}
\end{figure}
From the figure we easily see that the homeomorphisms are:
\[
\left\{ \begin{array}{llll}
\varphi_1(x,y): U_1 \longrightarrow \mathbb{R} & =y  &\Longrightarrow & \varphi_1^{-1}(y)=(\sqrt{1-y^2},y)\\
\varphi_2(x,y): U_2 \longrightarrow \mathbb{R} & =-y &\Longrightarrow & \varphi_2^{-1}(y)=(-\sqrt{1-y^2},y)\\
\varphi_3(x,y): U_3 \longrightarrow \mathbb{R} & =x  &\Longrightarrow & \varphi_3^{-1}(x)=(x,\sqrt{1-x^2})\\
\varphi_4(x,y): U_4 \longrightarrow \mathbb{R} & =-x &\Longrightarrow & \varphi_4^{-1}(x)=(x,-\sqrt{1-x^2})
\end{array}
\right.
\]
Then the coordinate transformations for the overlapping charts can be calculated.
\[
\left\{ \begin{array}{llll}
\Psi_{13} &=\varphi_1\circ\varphi_3^{-1}(x) &=\varphi_1(x,\sqrt{1-x^2})    &=\sqrt{1-x^2}\\
\Psi_{14} &=\varphi_1\circ\varphi_4^{-1}(x) &=\varphi_1(x,-\sqrt{1-x^2})   &=-\sqrt{1-x^2}\\
\Psi_{23} &=\varphi_2\circ\varphi_3^{-1}(x) &=\varphi_2(x,\sqrt{1-x^2})    &=-\sqrt{1-x^2}\\
\Psi_{24} &=\varphi_2\circ\varphi_4^{-1}(x) &=\varphi_2(x,-\sqrt{1-x^2})   &=\sqrt{1-x^2}
\end{array}
\right.
\]
From above we see that they are $C^\infty$--functions. And thus we have found the charts and coordinate transformations for $S^1$.
\end{example}

\begin{example}
$S^2$ is a complex manifold which is identified with the Riemann sphere $\mathbb{C}\cup\{\infty\}$.\\
The stereographic coordinates of a point $P(x,y,z)\in S^2\setminus\{ \mbox{ North Pole }\}$ projected from the north pole are:
\[ (X,Y)=\left( \frac{x}{1-z}, \frac{y}{1-z}\right). \]
While those of a point $P(x,y,z)\in S^2\setminus\{ \mbox{ South Pole }\}$ projected from the south pole are:
\[ (U,V)=\left( \frac{x}{1+z}, \frac{-y}{1+z}\right). \]

\begin{figure}[H]
\centering
\epsfig{file=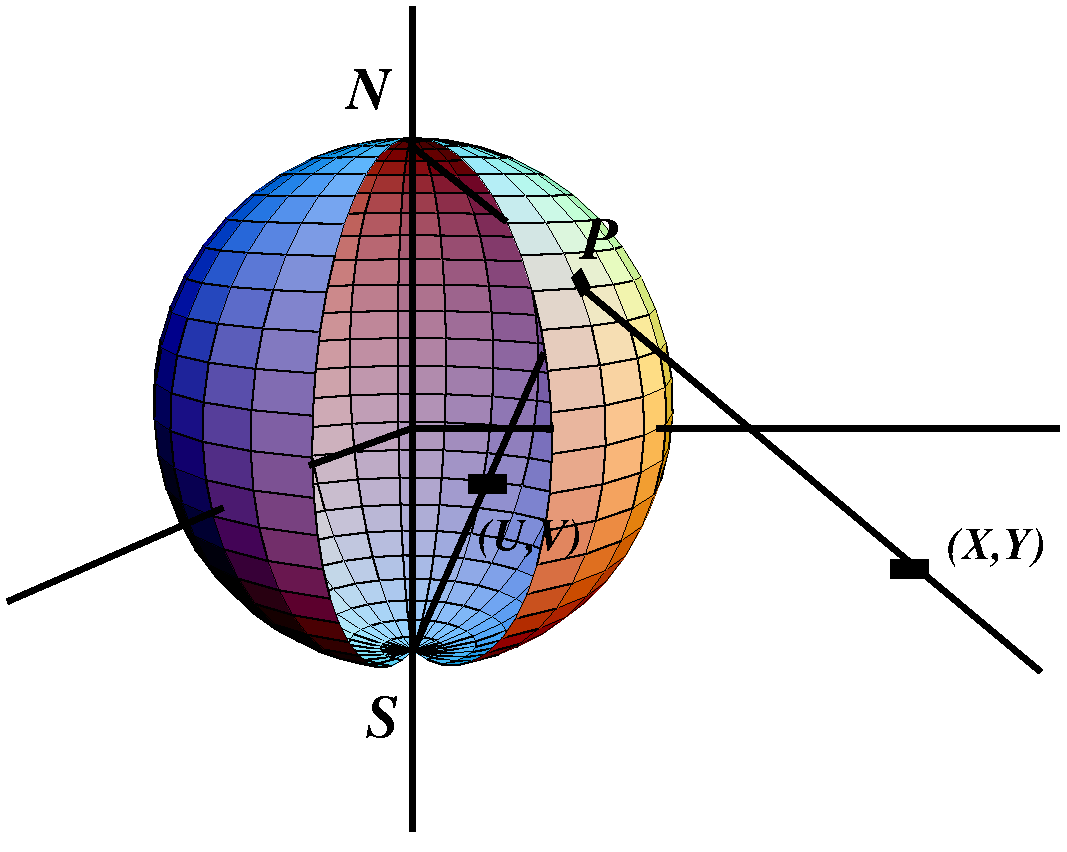, height=7cm}
\caption{A point $P$ projected from the south and the north pole.}
\label{fig:riemann}
\end{figure}

Let us define complex coordinates as 
\[Z=X+iY\mbox{ , } \overline{Z}=X-iY\mbox{ , }  W=U+iV\mbox{ , }  \overline{W}=U-iV.\]
Then 
\begin{align*}
Z &=\frac{x+iy}{1-z}=\frac{1+z}{1-z}\cdot\frac{x+iy}{1+z}=\frac{1+z}{1-z}(U-iV)\\
&=\frac{(1+z)(1+z)}{(1-z)(1+z)}(U+iV)=\frac{(1+z)^2}{1-z^2}(U+iV)\\
\intertext{Since we are on the unit sphere $x^2+y^2+z^2=1$ we have $1-z^2=x^2+y^2$ and hence we get:}
&=\frac{(1+z)^2}{x^2+y^2}(U+iV)=\frac{U+iV}{\frac{x^2+y^2}{(1+z^2)}}=\frac{U+iV}{U^2+V^2}=\frac{\overline{W}}{\overline{W} W}=\frac{1}{W}.
\end{align*}
Now we have shown that $Z=1/W$, thus $Z$ is a holomorphic function of $W$. Thus $S^2$ is a complex manifold which is identified with the Riemann sphere.
\end{example}

\begin{example}
The n--dimensional sphere $S^n$ is a differentiable manifold.\\
The sphere $S^n$ is realized in $\mathbb{R}^{n+1}$ as
\[ \sum_{i=0}^{n} (x^i)^2 = 1.\]
Introduce the coordinate neighborhoods

\begin{align*}
U_{i+} &=\{ (x^0, \ldots ,x^n)\in S^n : x^i > 0\}\\
U_{i-} &=\{ (x^0, \ldots ,x^n)\in S^n : x^i < 0\}.\\
\intertext{Define the coordinate maps $ \varphi_{i+} : U_{i+}\longrightarrow\mathbb{R}^n $ and $ \varphi_{i-} : U_{i-}\longrightarrow\mathbb{R}^n $ by}
\varphi_{i+}(x^0,\ldots ,x^n) &= (x^0,\ldots , x^{i-1},x^{i+1},\ldots , x^n)\\
\varphi_{i-}(x^0,\ldots ,x^n) &= (x^0,\ldots , x^{i-1},x^{i+1},\ldots , x^n)\\
\intertext{Note that the domains of $\varphi_{i+}$ and $\varphi_{i-}$ are different, and $\varphi_{i\pm}$ are projections of the hemisphere $U_{i\pm}$ to the plane $x^i =0$.}
\intertext{By $\sqrt{\ldots}$ we mean \[\sqrt{1-(x^{0})^2-\ldots -(x^{i-1})^2-(x^{i+1})^2-\ldots - (x^n)^2}.\] 
Then the inverses of the coordinate maps are:}
\varphi_{i+}^{-1}(x^0,\ldots , x^{i-1},x^{i+1},\ldots , x^n) &=(x^0, \ldots ,x^{i-1}, \sqrt{\ldots}\;, x^{i+1},\ldots , x^n)\\
\varphi_{i-}^{-1}(x^0,\ldots , x^{i-1},x^{i+1},\ldots , x^n) &=(x^0, \ldots ,x^{i-1},-\sqrt{\ldots}\;, x^{i+1},\ldots , x^n)\\
\intertext{Then the coordinate transformations for the charts $U_i$ and $U_j$ such that $U_i \cap U_j \neq \emptyset$ can be written as:}
\Psi_{i+j+}
 &= \varphi_{i+} \circ \varphi_{j+}^{-1}(x^0,\ldots , x^{i-1},x^{i+1},\ldots , x^n)\\
 &=\varphi_{i+}(x^0,\ldots ,x^{i-1}, \sqrt{\ldots}\;, x^{i+1},\ldots ,x^n)\\
 &=(\ldots ,x^{i-1},x^{i+1},\ldots ,x^{j-1},\sqrt{\ldots}\;, x^{j+1},\ldots )\\
 & \\
\Psi_{i+j-} 
 &= \varphi_{i+} \circ \varphi_{j+}^{-1}(x^0,\ldots , x^{i-1},x^{i+1},\ldots , x^n)\\
 &=\varphi_{i+}(x^0,\ldots ,x^{i-1},-\sqrt{\ldots}\;, x^{i+1},\ldots ,x^n)\\
 &=(\ldots ,x^{i-1},x^{i+1},\ldots ,x^{j-1},-\sqrt{\ldots}\;, x^{j+1},\ldots )
\end{align*}
And similarly we can obtain the different coordinate transformation functions and see that they are differentiable on the intersection of the charts, and hence $S^n$ is a differentiable manifold.
\end{example}

We all know how important it is to be able to construct subspaces in ordinary linear algebra. The fact that one can ``sort'' a space in parts of smaller dimensions, is rather trivial when we deal with Euclidean geometry. Take for example the ``three--dimensional'' sphere. It is quite easy to convince anyone that we can cut it into ``two--dimensional'' circles, which we describe by a plane structure. The same procedure should be valid for manifolds. We should be able to describe a part of an $n$--dimensional manifold as an $k$--dimensional submanifold. One way of looking at submanifolds is considering a $k$--dimensional submanifold of an $n$--dimensional manifold \man{M} as a subset \man{K} of \man{M} which is a $k$--dimensional manifold in the induced topology. We will now try to define submanifolds formally. For this purpose we first introduce a valuable tool, The Rank Theorem.

\begin{theorem}
\textbf{(The Rank Theorem)} Let $\Omega \subset \mathbb{R}^{m}$ be open and $f \in C^{k}(\Omega , \mathbb{R}^{n})$ i.e. $f:\Omega \longrightarrow \mathbb{R}^{n}$ be a $C^{k}$--map. Suppose that $rank \; d_{a} f =r$, $(\forall a \in \Omega )$. Then $\forall a \in \Omega $ $\exists $ a neighborhood $U$ of $a$, $\exists $ a neighborhood $V$ of $f(a)$, $\exists $ a neighborhood $G$ of $0 \in \mathbb{R}^{m} $, $\exists $ a neighborhood $D$ of $0 \in \mathbb{R}^{n}$. Furthermore, there exists diffeomorphisms $u:G\longrightarrow U$ and $v:V\longrightarrow D$ such that for every set $(x_{1},\ldots ,x_{n})\in G$
\[
(v \circ f \circ u)(x_{1},\ldots ,x_{n})=(x_{1},\ldots ,x_{r},0,\ldots ,0).
\]

\begin{figure}[H]
\centering
\epsfig{file=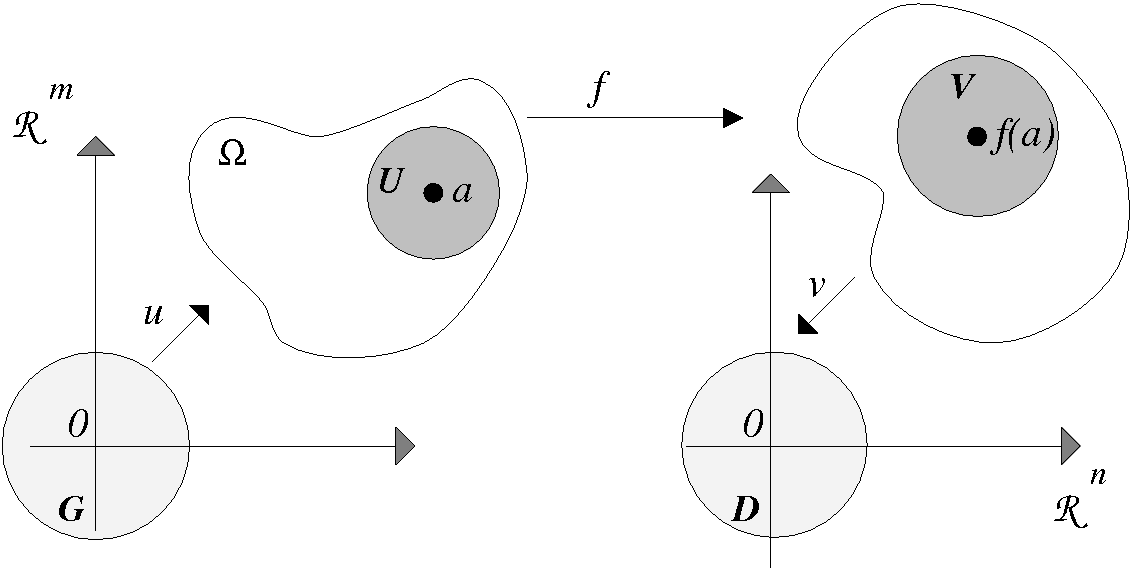, height=6cm , width=11cm}
\caption{A visualization of the mappings in The Rank Theorem.}
\label{fig:rank}
\end{figure}
\end{theorem} 

\begin{proof}
We may suppose that $a=f(a)=0$ and that $d_{0}f$ is the map $(v_{1},\ldots ,v_{n}) \longmapsto (v_{1},\ldots ,v_{r},0,\ldots ,0)$ which satisfies 

\begin{equation}
\label{eq:detfunctionalmatrix}
det \, \left[ \frac{\partial f_{i}}{\partial x_{j}}(0)\right]_{1\leq i,j\leq r}\neq 0.
\end{equation}
Then we define the $C^{k}$--map $h:\Omega \longrightarrow \mathbb{R}^{m}$ by $h(x)=(f_{1}(x),\ldots ,f_{r}(x),x_{r+1},\ldots x_{m})$ where $x \in \Omega$ and $f(x) = (f_{1}(x),\ldots ,f_{r}(x),\ldots f_{n}(x))$. Then the Jacobi matrix of $h$ is the $(m \times m)$ matrix
\[
d_{0}h = \left[ \frac{\partial h_{i}}{\partial x_{j}}(0) \right] 
= 
\left[ \begin{array}{ll}
\left[ \frac{\partial f_{j}}{\partial x_{k}}(0) \right]_{ 1\leq j \leq r \atop 1\leq k \leq m} &\left\{ \begin{array}{l} \mbox{(r)--rows} \\ \mbox{(m)--columns} \end{array} \right. \\
 \\
\underbrace{\begin{array}{ccc} 0 & \ldots & 0 \\ \vdots &  & \vdots \\ 0 & \ldots & 0 \end{array}}_{\mbox{(r)--columns}} & \underbrace{\begin{array}{ccc} 1 & \ldots & 0 \\ \vdots & \ddots & \vdots \\ 0 & \ldots & 1 \end{array}}_{\mbox{(m-r)--columns}}
\end{array} \right].
\]
Since the determinant in equation (\ref{eq:detfunctionalmatrix}) is equal to zero, we conclude that the Jacobi matrix of $h$ has full rank, that is $det \, d_{0}h \neq 0$. \\
Then by the Inverse Mapping Theorem $\exists$ $U'$ and $G'$ neighborhoods of $0 \in \mathbb{R}^{m}$ such that $h|_{U'}$ is a $C^{k}$--diffeomorphic map from $U'$ to $G'$. By the same theorem we know that ${h^{-1}|_{G'}}$ is a $C^{k}$--diffeomorphism. Moreover, for all $y\in G'$ we have

\begin{align}
\label{eq:hinvers}
h\circ h^{-1}(y) & = (h_{1}(h^{-1}(y)),\ldots ,h_{m}(h^{-1}(y))) \nonumber \\
 & =(f_{1}(h^{-1}(y)),\ldots ,f_{r}(h^{-1}(y)),y_{r+1},\ldots ,y_{m}), 
\end{align} 
and hence $(y_{1},\ldots ,y_{r})= (f_{1}(h^{-1}(y)),\ldots ,f_{r}(h^{-1}(y)))$. \\
Now we define $g=f \circ h^{-1}$, and comparing to (\ref{eq:hinvers}) we get

\[
g(x)= (x_{1},\ldots ,x_{r},g_{r+1}(x),\ldots ,g_{n}(x))
\]
where $g_{r+1},\ldots ,g_{n}\in C^{k}(G')$. We can express the Jacobi matrix of $g$ at $0$ as the following $(n\times m)$ matrix
\[
d_{0}g  = \left[ \begin{array}{ll}
\begin{array}{ccc} 1  & \ldots & 0 \\ \vdots & \ddots   & \vdots \\ 0 & \ldots & 1 \end{array}  & \begin{array}{ccc} 0 & \ldots & 0 \\ \vdots &  & \vdots \\ 0 & \ldots & 0 \end{array}\\
 \\
\begin{array}{ccc} \ast & \ldots & \ast \\ \vdots &  & \vdots \\ \ast & \ldots & \ast \end{array}  &  \left[ \frac{\partial g_{i}}{\partial x_{j}}(0) \right]_{r+1 \leq i \leq n \atop r+1 \leq j \leq m}
\end{array} \right],
\]
where the blocks have the following sizes:
\begin{figure}[H]
\centering
\epsfig{file=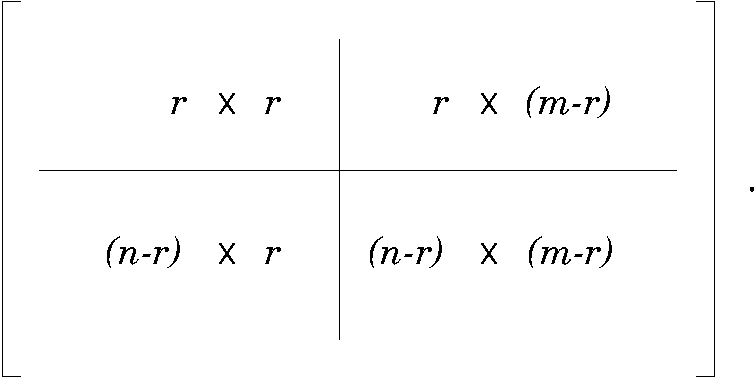, height=3cm , width=6cm}
\end{figure}
Since we know that $rank \, d_{a}f =r$ and $g=f\circ h^{-1}$, we conclude that $rank\, d_{x}g = \, rank\, d_{h^{-1}}f =r$ for all $x\in G'$ and therefore we obtain that the Jacobi matrix of $g$ at $x$ is equal to zero, so $g_{r+1},\ldots ,g_{n}$ are independent of $x_{r+1},\ldots ,x_{m} \in G'$. \\
We let $G'\subset \mathbb{R}^{m}$ have the following structure
\[
G'\; = \; {G'}^{r} \, \times \, {G'}^{m-r},
\]
where ${G'}^{r}$ and ${G'}^{m-r}$ are open neighborhoods of $0$ in \real{r} and \real{m-r} respectively. Then we define the map
\begin{align*}
v(y_{1},\ldots ,y_{n}) & =(y_{1},\ldots ,y_{r}, y_{r+1}-g_{r+1}(y_{1},\ldots ,y_{r},0,\ldots ,0),\ldots  \\
                       & \qquad , y_{n}-g_{n}(y_{1},\ldots ,y_{r},0,\ldots ,0)).
\end{align*}
This makes $g_{r+1},\ldots ,g_{n}$ disappear when we compose it with $f\circ h^{-1}$. The Jacobi matrix of $v$ at $0$ is the following $(n \times n)$ matrix

\[
d_{0}v  = \left[ \begin{array}{cc}
\begin{array}{ccc} 1 & \ldots & 0 \\ \vdots & \ddots & \vdots \\ 0 & \ldots & 1 \end{array} &        \begin{array}{ccc} 0 & \ldots & 0 \\ \vdots &  & \vdots \\ 0 & \ldots & 0 \end{array} \\
 \\ 
\begin{array}{ccc} \ast & \ldots & \ast \\ \vdots &  & \vdots \\ \ast & \ldots & \ast \end{array} &  \begin{array}{ccc} 1 & \ldots & 0 \\ \vdots & \ddots  & \vdots \\ 0 & \ldots & 1 \end{array} 
\end{array} \right], 
\]
where the blocks have the following sizes:
\begin{figure}[H]
\centering
\epsfig{file=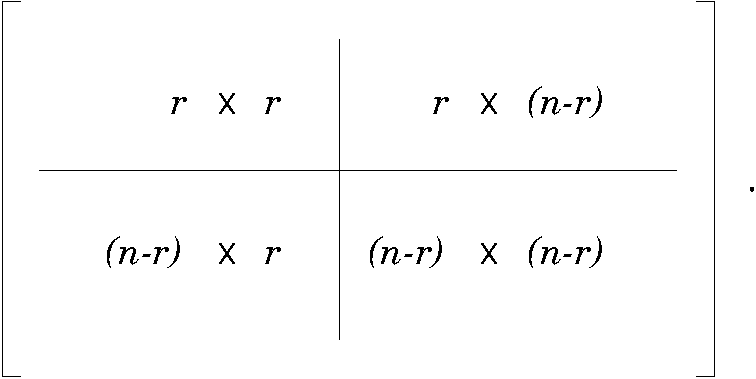, height=3cm , width=6cm}
\end{figure}
Applying the Inverse Mapping Theorem, we see that $\exists $ a neighborhood $V$ of $0 \in \mathbb{R}^{n}$ and $\exists $ a neighborhood $D$ of $0 \in \mathbb{R}^{n}$ such that $v|_{V} : V \longrightarrow D$ is a $C^{k}$--diffeomorphism for a neighborhood $G$ of $0 \in \mathbb{R}^{n}$ such that $G \subset G'$ and $g(G) \subset V$. We define $U=h^{-1}(G)$ and $u=h^{-1}|_{G}$, then $u:G\longrightarrow U$ is a $C^{k}$--diffeomorphism. Then for $x\in G$ we have 
\begin{align*}
v\circ f\circ u (x) & = v\circ f\circ h^{-1} (x)=v(g(x)) \\
                    & = v(x_{1},\ldots ,x_{r},g_{r+1}(x),\ldots ,g_{n}(x)) \\
                    & = v(x_{1},\ldots ,x_{r},g_{r+1}(x)-g_{r+1}(x_{1},\ldots ,x_{r},0,\ldots ,0),\ldots \\
                    & \qquad ,g_{n}(x)-g_{n}(x_{1},\ldots ,x_{r},0,\ldots ,0)).
\end{align*}
But $g_{r+1}(x),\ldots ,g_{n}(x)$ are independent of $x_{r+1},\ldots ,x_{n}$ in $G'$ which implies that this holds also in $G \subset G'$, and thus
\[
v\circ f\circ u (x)\, = \, v\circ f\circ u (x_{1},\ldots ,x_{n})\, = \, (x_{1},\ldots , x_{r},0,\ldots ,0). \quad \blacksquare
\] 
\end{proof}

\begin{theorem}
\label{thm:subman}
Let $k\in \mathbb{N}$ such that $1 \leq k \leq n$, and let $\mathcal{M}\subset \mathbb{R}^{n}$. Then the following conditions are equivalent.
\begin{description}
\item[\textbf{(1)  }] $\forall a \in \mathcal{M} \, \exists$ a neighborhood $U$ of $a$ and a function $f\in C^{\infty}(U,\mathbb{R}^{n-k})$ such that $rank \, d_{a}f =n-k$ and $U \cap \mathcal{M} = f^{-1}(0)$.
\item[\textbf{(2)  }] $\forall a \in \mathcal{M} \, \exists$ a neighborhood $U$ of $a$, an open set $V\subset \mathbb{R}^{n}$ and a diffeomorphism $\phi :U\longrightarrow V$ such that $\phi (U\cap \mathcal{M}) = V\cap (\mathbb{R}^{k} \times \{0 \})$.   
\item[\textbf{(3)  }] $\forall a \in \mathcal{M} \, \exists$ a neighborhood $U$ of $a$, an open set $W\subset \mathbb{R}^{k}$, a function $\psi \in C^{\infty}(W,\mathbb{R}^{n})$ such that 
\[
rank \, d_{\psi^{-1}(a)}\psi =k,
\]
and $\psi : W \longrightarrow \mathcal{M} \cap U$ is bijective and $\psi^{-1}$ is continuous.
\item[\textbf{(4)  }] After permuting the variables in $\mathbb{R}$, locally, \man{M} is the graph of a $C^{\infty}$--mapping from an open subset of $\mathbb{R}^{k}$ into $\mathbb{R}^{n-k}$.
\end{description}
\end{theorem}

\begin{proof}
\begin{description}
\item[$\mathbf{1 \Longrightarrow 2  }$]  We know that the Jacobi matrix of $f$ has constant rank equal to $n-k$ in a neighborhood of $a\in \mathcal{M}$. We may suppose that $U$ is this this neighborhood. Then by The Rank Theorem we can find diffeomorphisms $u$ and $v$ such that
\[
(v\circ f\circ u)(x_{1},\ldots ,x_{n})=(0,\ldots ,0,x_{k+1},\ldots ,x_{n}),
\]
which can be seen as the projection onto the $n-k$ dimensional subspace of $\mathbb{R}^{n}$, in a neighborhood of $0$. We put $u(0)=a$. Then we let $U'\subset U$ be the range of $u$ and define $\phi = u^{-1}$. Illustrating what we have accomplished so far, hopefully gives a better understanding.

\begin{figure}[H]
\centering
\epsfig{file=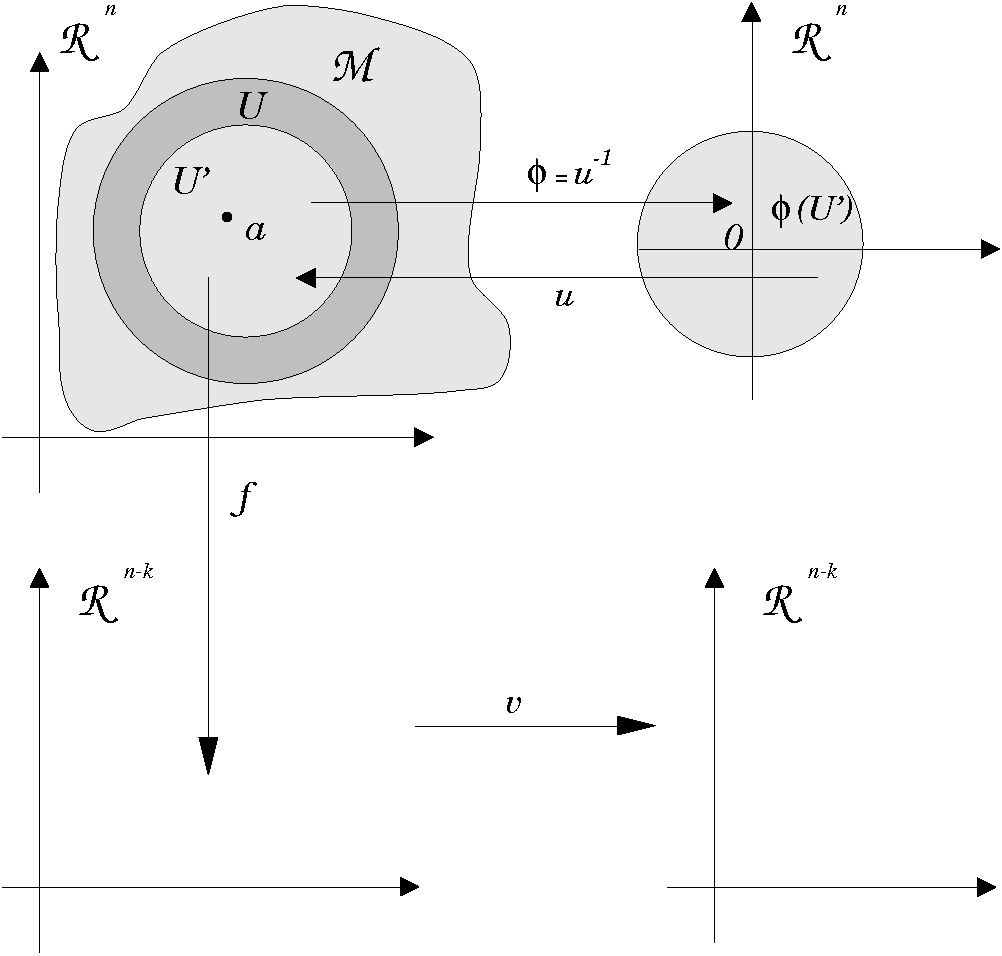, height=10cm, width=12cm}
\caption{An illustration of the mappings presented above.}
\end{figure}
According to the assumption in (1), we get
\begin{equation}
\label{eq:fiUnM}
\phi (U' \cap \mathcal{M}) = \phi (U' \cap (U \cap \mathcal{M}))=\phi (U' \cap f^{-1}(0))= \phi (U')\cap \phi(f^{-1}(0)).
\end{equation}
We also know that $f^{-1}(a')$ is a mapping from $a' \in \mathbb{R}^{n-k}$ into $\mathbb{R}^{n}$ and v is a mapping from a neighborhood of $0\in \mathbb{R}^{n-k}$ to a neighborhood of $a'$. Since we have $a'=0$ in equation (\ref{eq:fiUnM}), we can write $(v\circ f)^{-1}(0)$ instead of $f^{-1}(0)$ and then we get
\[
\phi (U' \cap \mathcal{M})= u^{-1}(v\circ f^{-1})(0) \cap \phi (U').
\]
Using the composition law and the fact that $V$ is equal to $\phi (U')$, we get 
\[
\phi (U' \cap \mathcal{M})=V\cap (v\circ f\circ u^{-1})(0)=V\cap (\mathbb{R}^{k} \times \{ 0 \}),
\]
and thus we have constructed (2) from (1).
\item[$\mathbf{2 \Longrightarrow 3  }$]  Set $I(x)=(x,0)$, then $I:\mathbb{R}^{k} \longrightarrow \mathbb{R}^{n}$. This enables us to look at \real{k} from \real{n}. Then we define $W=I^{-1}(V\cap (\mathbb{R}^{k} \times \{0 \}))$ where $V$ is as in (2), and $\psi = \phi^{-1} \circ I$ where $\phi$ is also as in (2) which implies that $\phi^{-1} : V \longrightarrow U$ and $\psi : W \longrightarrow \mathbb{R}^{n}$which is bijective. Note that since $\phi$ and $I$ are diffeomorphic, we find out that $\psi^{-1}=I^{-1}\circ \phi$ is continuous. \\
 \\
\item[$\mathbf{3 \Longrightarrow 4  }$]  Split \real{n} into \real{k} and \real{n-k} and let $\psi$ map \real{k} into the \real{n-k} part of \real{n}, since we can do the permutation as we please, there will not be any complication. Then \man{M} is the graph of the $C^{\infty}$--mapping $\psi$.
\item[$\mathbf{4 \Longrightarrow 1  }$]  According to (4), $\mathcal{M}$ is the graph of a $C^{\infty}$--mapping from \real{k} into \real{n-k}. Thus we can formally write
\[
\mathcal{M}= \{ (x,F(x)): x\in \mathcal{E} \subset \mathbb{R}^{k} \mbox{   where $\mathcal{E}$ is the domain of $F(x)$.} \}
\]
Then we start with looking at \real{n} as $\mathbb{R}^{k} \times \mathbb{R}^{n-k}$ and consider open subsets $U\subset \mathbb{R}^{k}$ and $V\subset \mathbb{R}^{n-k}$, hence we can say that$F\in C^{\infty}(U,V)$ and
\[
\mathcal{M} \cap (U\times V)=\mbox{graph }F.
\]
Then we define $f(x,y)=y-F(x)$ and since $y\in V$ and $x\in U$, we may vary $x$ and $y$ independently, thus $y$ is not always equal to $F(x)$ and hence $rank f(x,y)=n-k$. Furthermore we know that each element of $U\cap \mathcal{M}$ can be written in the form $(x,F(x))$ which tells us that the image of $U\cap \mathcal{M}$ under $f(x,y)$ is
\[
f(x,y)=y-F(x)=F(x)-F(x)=0. \quad \blacksquare
\]  
\end{description}
\end{proof}

\begin{defn}
Let \man{N} be a $C^{\infty}$--manifold with an atlas \man{U}. We say that $\mathcal{M}\subset\mathcal{N}$ is a $k$--dimensional submanifold of \man{N} if
\begin{align*}
&\forall a \in \mathcal{M} \, \exists \, (U,\phi)\in \mathcal{U} \mbox{  such that  } a\in U \\
\intertext{and}
&\phi (U\cap \mathcal{M}) = \phi (U) \cap (\mathbb{R}^{k} \times \{0 \}).
\end{align*}
\end{defn}

\begin{remark}
In particular a subset \man{M} of \real{n} satisfying any of the equivalent conditions in Theorem \ref{thm:subman} is called a $k$--dimensional submanifold of \real{n} ( of class $C^{\infty}$ ). \\
\end{remark}

\subsection{Tangent vectors and Tangent spaces}
In general an elementary picture of a vector as an arrow connecting a point and the origin does not work in a manifold. Where is the origin? What is a straight arrow? How do we define a straight arrow that connects two points on a curved surface?\\
The notion of a vector tangent to a curve or a surface in \real{n} is intuitively clear. But if we try to generalize the notion of tangent vector, we face a difficulty:\\
The elementary definition makes a tangent vector to a surface in \real{n} a tangent vector to \real{n}. But an arbitrary manifold does not have to be contained in any Euclidean space, so we need a definition of a tangent vector that does not depend on any such assumption. We start with a look at the notion of differentiable functions on manifolds.\\   
The definition of differentiability of a real--valued function on a differentiable manifold \man{M} is almost the same as the corresponding definition on \real{n}.
  
\begin{defn}
\label{sec:diff}
Let $ f : \mathcal{W} \subset \mathcal{M} \longrightarrow \mathbb{R}$ be a function on an open subset \man{W} of a differentiable manifold \man{M}. We say that $f$ is differentiable at a point $p\in\mathcal{W}$, provided that for some chart $(U,\varphi)$ such that $p\in U \subset\mathcal{W}$, the composition $f\circ\varphi^{-1}(U):\mathbb{R}^{n} \longrightarrow \mathbb{R}$ is differentiable in the ordinary Euclidean sense at $\varphi (p)$. If $f$ is differentiable at all points of \man{W}, we say that f is differentiable on \man{W}.
\begin{figure}[H]
\centering
\epsfig{file=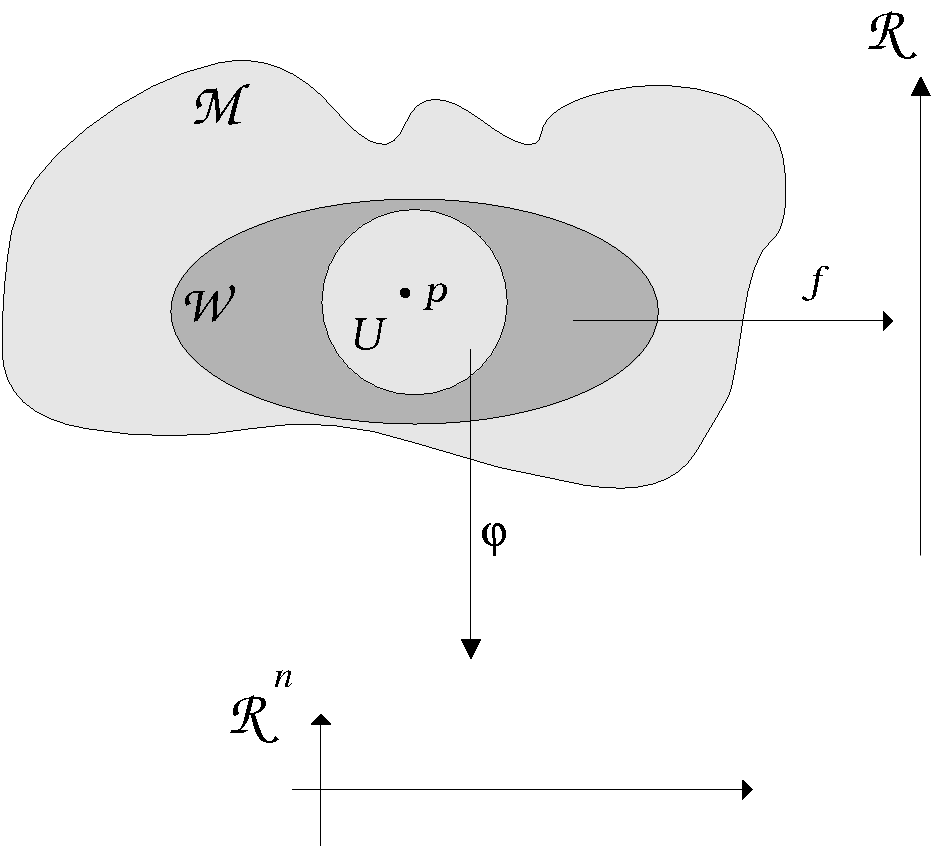, height=6cm, width=8cm}
\caption{A visualization of definition \ref{sec:diff}.}
\end{figure} 
\end{defn}

\begin{lemma}
The definition of differentiability of a real--valued function on a differentiable manifold \man{M} does not depend on the choice of the chart.
\end{lemma}

\begin{proof}
If $(U_1, \varphi_1)$ and $(U_2, \varphi_2)$ are charts on a differentiable manifold \man{M}, then the change of coordinates $\varphi_1\circ\varphi_2^{-1}$ is differentiable by definition \ref{sec:atlas}. We can write $f\circ\varphi_{2}^{-1}$ as 
\[f\circ\varphi_2^{-1}=(f\circ\varphi_1^{-1})\circ(\varphi_1\circ\varphi_2^{-1}).\]
Since the composition of the Euclidean--differentiable function is differentiable, it follows that the differentiability of $f\circ\varphi_1^{-1}$ implies the differentiability of $f\circ\varphi_2^{-1}$, and vice versa.\quad $\blacksquare$
\end{proof}

\begin{defn}
Let \man{M} be a differentiable manifold. We denote by 
\[
\mathcal{F}(\mathcal{M})=\{ f: \mathcal{M}\longrightarrow\mathbb{R} : f \mbox{ is differentiable }\}
\] 
the algebra of real--valued differentiable functions. Then for $\delta , \rho \in\mathbb{R}$ and $ f,g\in\mathcal{F}(\mathcal{M})$ the following combinations are defined by:
\begin{align*}
(\delta f + \rho g )(p) &=\delta f(p) + \rho g(p)\\
\intertext{and}
(fg)(p) &=f(p)g(p)
\end{align*}
for any point $p\in\mathcal{M}$. Also we identify any $\delta\in\mathbb{R}$ with the constant function $\delta$ given by $\delta (p)=\delta $ for $p\in\mathcal{M}$.
\end{defn}
  
\begin{defn}
\label{sec:tanvec}
Let $p$ be any point on a differentiable manifold \man{M}. A tangent vector, denoted by $v_p$ to \man{M} at $p$ is a real--valued function $v_p:\mathcal{F}(\mathcal{M})\longrightarrow \mathbb{R}$, such that
\begin{align}
\mbox{\underline{Linear property:}}&\mbox{ } &v_p[\delta f +\rho g]=\delta v_p[f]+\rho v_p[g]\\
\mbox{\underline{Leibnitz property:}}&\mbox{ } &v_p[fg]=f(p)v_p[g]+g(p)v_p[f]
\end{align}
is satisfied for all $\delta, \rho\in\mathbb{R}$ and $f,g\in\mathcal{F}(\mathcal{M})$.
\end{defn}

\begin{defn}
Let $(U,\varphi)$ be a chart on a differentiable manifold \man{M} with the local coordinate system $\varphi = (x^{1}, \ldots ,x^{n})$. The natural coordinate functions of \real{n} are denoted by $u_i$. For $f\in\mathcal{F}(\mathcal{M})$ and $p\in(U, \varphi)$, we write $q=\varphi(p)$ and define
\begin{equation}
\left.\frac{\partial f}{\partial {x^i}}\right|_p = \left.\frac{\partial f}{\partial {x^i}}\right|_{\varphi^{-1}(q)} =\left. \frac{\partial (f\circ\varphi^{-1})}{\partial u_i}\right|_q \mbox{      for $i=1,\ldots ,n$.}
\end{equation}
\end{defn}

\begin{remark}
The derivative that appears in the right hand side of the equation is the ordinary Euclidean partial derivative.
\end{remark}

\begin{lemma}
Let $(U,\varphi)$ be a chart on a differentiable manifold \man{M} and let $\varphi = (x^{1}, \ldots ,x^{n})$. If $p\in (U)$, then
\[
\left. \frac{\partial}{\partial {x^i}}\right|_p : \mathcal{F}(\mathcal{M}) \longrightarrow \mathbb{R}
\]
defined by 
\[\left. \frac{\partial}{\partial {x^i}}\right|_p (f)=\left. \frac{\partial f}{\partial {x^i}}\right|_p \]
is a tangent vector to \man{M} at $p$ for $i=1,\ldots ,n$.
\end{lemma}

\begin{proof}
Put $q=\varphi (p)$ and let $\delta,\rho\in\mathbb{R}$ and $f,g\in\mathcal{F}(\mathcal{M})$. Then checking the properties of tangent vector, we get:

\begin{align*}
\left.\frac{\partial {(\delta f+\rho g)}}{\partial {x^i}}\right|_p &=\left.\frac{\partial }{\partial u_i}((\delta f+\rho g)\circ \varphi^{-1})\right|_q =\left.\frac{\partial }{\partial u_i}(\delta (f\circ \varphi^{-1})+\rho (g\circ \varphi^{-1}))\right|_q \\
 \\
&=\left. \delta \frac{\partial {(f\circ \varphi^{-1})}}{\partial u_i}\right|_q+\left. \rho \frac{\partial {(g\circ \varphi^{-1})}}{\partial u_i}\right|_q =\delta \left.\frac{\partial f}{\partial {x^i}}\right|_p + \rho \left.\frac{\partial g}{\partial {x^i}}\right|_p . \\
\intertext{Hence the Linear property is satisfied.}
\left.\frac{\partial {(f g)}}{\partial {x^i}}\right|_p &=\left.\frac{\partial }{\partial u_i}((f g)\circ \varphi^{-1})\right|_q=\left.\frac{\partial }{\partial u_i}((f\circ \varphi^{-1})(g\circ \varphi^{-1}))\right|_q \\
 \\
&=(g\circ\varphi^{-1})(q)\left. \frac{\partial{(f\circ\varphi^{-1})}}{\partial {u_i}}\right|_q +(f\circ\varphi^{-1})(q)\left. \frac{\partial {(g\circ\varphi^{-1})}}{\partial {u_i}}\right|_q \\
 \\
&=g(p)\left. \frac{\partial f}{\partial {x^i}}\right|_p + f(p)\left. \frac{\partial g}{\partial {x^i}}\right|_p . \\
\end{align*}
The Leibnitz property is satisfied as well, and our proof is finished.   $\blacksquare$
\end{proof}

\begin{theorem}
\label{sec:chain}
\textbf{ (The Chain Rule) } Let $(U_1 ,\varphi_1)$ and $(U_2 ,\varphi_2)$ be two charts on a differentiable manifold \man{M}. Let \coord be a local coordinate system on $(U_1 ,\varphi_1)$ and $(y^1,\ldots ,y^n)$ be local coordinate system on $(U_2 ,\varphi_2)$, then for $f\in\mathcal{F}(\mathcal{M})$ on $U_1\cap U_2$ we have:
\[ \frac{\partial f}{\partial {y^i}} =\sum_{i=1}^{n} {\frac{\partial f}{\partial {x^j}}\frac{\partial {x^j}}{\partial {y^i}}}.\]
\end{theorem}

\begin{defn}
Let \man{M} be a differentiable manifold. The tangent space to \man{M} at a point $p\in\mathcal{M}$ is the set of all tangent vectors to \man{M} at $p$. We denote it by \tpm. Formally we can define \tpm as follows:

\begin{align*}
\mathcal{T}_p\mathcal{M} & = \{ v_p : \mathcal{F}(\mathcal{M})  \longrightarrow  \mathbb{R} : v_p [\delta f+\rho g] = \delta v_p [f] +\rho v_p [g] \mbox{ and } \\
 \\
& v_p [fg] = f(p) v_p [g] + g(p) v_p [f] \mbox{ for } \delta ,\rho\in\mathbb{R}\mbox{ and } f,g \in \mathcal{F}(\mathcal{M})\}.
\end{align*}
\end{defn}

\begin{theorem}
\label{sec:tpmisofdimn}
If \man{M} is an $n$--dimensional differentiable manifold, then the tangent space \tpm  is also of dimension $n$.
\end{theorem}

\begin{proof}
What we need to show is that we can write a tangent vector as the linear combination of $n$ tangent vectors, and that the $n$ vectors are linearly independent. We can express the linear combination as follows:

\begin{equation}
\label{eq:lincomb}
v_p =\sum_{i=1}^{n} {v_p[x^i]\left.\frac{\partial}{\partial {x^i}}\right|_p }
\end{equation}
if $(U,\varphi)$ is a chart with $p\in U$ and $\varphi = (x^{1},\ldots ,x^{n})$. \\
First we prove the existence of a tool that we need.
We let $g:\mathbb{R}^n \longrightarrow \mathbb{R}$ be a differentiable function, and let $a=(a_1,\ldots ,a_n)$ be in \real{n}. Now we will show that there exist differentiable functions $g_i :\mathbb{R}^n \longrightarrow \mathbb{R}$ (for $i=1,\ldots ,n$.) such that

\begin{equation}
\label{eq:g}
g=\sum_{i=1}^{n}{(u_i - a_i)g_i} + g(a)
\end{equation}
Where the set $(u_1,\ldots ,u_n)$ is the natural coordinate functions of \real{n}.\\
Now we fix $t=(t_1,\ldots ,t_n)\in \mathbb{R}^n$ and consider $f:\mathbb{R}\longrightarrow\mathbb{R}$ defined by

\begin{equation*}
f(s)=g(st_1 + (1-s) a_1,\ldots ,st_n+(1-s)a_n)=g(st+(1-s)a).
\end{equation*}
Then $f$ is differentiable, since $g$ is. Moreover we have $u_i(s)=st_i+(1-s)a_i$, and by the chain rule we get:

\begin{equation}
\label{eq:fprim}
f'(s)=\sum_{i=1}^{n}{\frac{\partial g}{\partial {u_i}}\frac{\partial {u_i}}{\partial {s_i}}}=\sum_{i=1}^{n}{(t_i-a_i)\frac{\partial g}{\partial {u_i}}(st_1 + (1-s) a_1,\ldots ,st_n+(1-s)a_n)}.
\end{equation}
On the other hand by the fundamental theorem of analysis we have:

\begin{equation}
\label{eq:integral}
\int_{0}^{1} {f'(s)ds} =f(1)-f(0)=g(t)-g(a).
\end{equation}
So from the equations (\ref{eq:fprim}) and (\ref{eq:integral}) we get:

\begin{equation*}
\begin{split}
g(t)-g(a) &=\int_{0}^{1}{f'(s)\,ds} \\
 \\
& =\int_{0}^{1}{\sum_{i=1}^{n}{(t_i-a_i)\frac{\partial g}{\partial {u_i}}(st_1 + (1-s) a_1,\ldots ,st_n+(1-s)a_n)}\,ds}\\
 \\
& = \sum_{i=1}^{n}{(t_i-a_i)\int_{0}^{1}{\frac{\partial g}{\partial {u_i}}(st_1 + (1-s) a_1,\ldots ,st_n+(1-s)a_n)\,ds}}.
\end{split}
\end{equation*}
If we denote $\int_{0}^{1}{\frac{\partial g}{\partial {u_i}}(st_1 + (1-s) a_1,\ldots ,st_n+(1-s)a_n)\,ds}$  by  $g_i(t)$, we obtain

\begin{align*}
g(t)-g(a) &=\sum_{i=1}^{n}{(u_i(t)-a_i)g_i(t)} \quad \Longrightarrow\\ 
g &=\sum_{i=1}^{n}{(u_i - a_i)g_i} + g(a)
\end{align*}
And so we have obtained (\ref{eq:g}). Now we are ready to show that 
\[
v_p =\sum_{i=1}^{n} {v_p[x^i]\left.\frac{\partial}{\partial {x^i}}\right|_p .}
\]
for all $f\in \mathcal{F}(\mathcal{M})$ and $p\in (U)$, where $(U,\varphi)$ is a chart  on a differentiable manifold \man{M} and $\varphi = (x^{1},\ldots ,x^{n})$.\\
We let $g:\mathbb{R}^n \longrightarrow \mathbb{R}$ be a globalization of $f\circ\varphi^{-1} : \varphi (U) \longrightarrow \mathbb{R}$, equivalently $g|_{\varphi (U)} = f\circ\varphi^{-1}$. Since we are only concerned with what is happening in an arbitrarily small neighborhood of $p$, we may suppose --- by taking a smaller $U$ --- that such a globalization exists. Then there exist  functions $g_i :\mathbb{R}^n \longrightarrow \mathbb{R}$ for $i=1,\ldots ,n.$ such that:

\begin{equation*}
g=\sum_{i=1}^{n} {(u_i - u_i(\varphi (p))) g_i} + g(\varphi (p)).
\end{equation*}
Combining this with the fact that $g=f\circ\varphi^{-1}$ on $U$, we see that 
\[
g(\varphi (p))= f\circ\varphi^{-1} (\varphi (p))=f(p).
\]
We see also that $ u_i(\varphi (p))= x^{i}(p)$, and if we let $g_i \circ\varphi$ be denoted by $f_i$ we get

\begin{equation}
\label{eq:f}
f=\sum_{i=1}^{n}{(x^i -x^i (p))f_i}+f(p).
\end{equation}
Now we are able to compute $v_p[f]$ from (\ref{eq:f}):

\begin{equation*}
\begin{split}
v_p[f] &=v_p\left[ \sum_{i=1}^{n}{(x^i -x^i(p))f_i} +f(p)\right]=\sum_{i=1}^{n}{v_p[f]\left[ (x^i -x^i(p))f_i \right]+v_p[f(p)]}\\
&=\sum_{i=1}^{n}{\left( f_i(p) v_p \left[ (x^i -x^i(p))\right]+(x^i(p)-x^i(p))v_p[f_i]\right) }+v_p[f(p)].
\end{split}
\end{equation*}
Since $x^i(p)-x^i(p)=0$ and $f(p)$ and $x^i(p)$ are constants we get:
\[ v_p[f]=\sum_{i=1}^{n} {f_i (p) v_p [x^i]} \]
which can be written as follows:

\begin{equation}
\label{eq:fpartial}
\left. \frac{\partial f}{\partial {x^j}}\right|_p =\sum_{i=1}^{n} {f_i(p) \left. \frac{\partial {x^i}}{\partial {x^j}}\right|_p}.
\end{equation}
And this is equal to $f_j(p)$ since $\left. \frac{\partial {x^i}}{\partial {x^j}}\right|_p$ is equal to the \emph{Kronecker delta function} defined by

\begin{equation*} 
\delta_{ij}=\left\{ \begin{array}{ll}
 1\quad \mbox{ for } &i=j\\
 0\quad \mbox{ for } &i\neq j
\end{array}\right.
\end{equation*}
and since
\begin{equation*}
\left.\frac{\partial {x^i}}{\partial {x^j}}\right|_p =\left.\frac{\partial {x^i \circ\varphi^{-1}}}{\partial {u_j}}\right|_{\varphi(p)} =\left.\frac{\partial {u_i}}{\partial {u_j}}\right|_{\varphi(p)}=\delta_{ij},
\end{equation*}
the equation (\ref{eq:fpartial}) is equal to $f_j(p)$, and hence

\begin{equation*}
\begin{split}
\left. \left( \sum_{j=1}^{n}{v_p[x^j] \frac{\partial}{\partial {x^j}}}\right)\right|_p &=\sum_{j=1}^{n} {v_p[x^j]\left. \frac{\partial f}{\partial {x^j}}\right|_p}\\
&= \sum_{j=1}^{n} {v_p[x^j] f_j(p)} = v(p)[f].
\end{split}
\end{equation*}
Now we have shown that the tangent vector $v_p[f]$ can be written as a linear combination of the vectors $\left. \frac{\partial}{\partial {x^1}} \right|_p ,\ldots ,\left. \frac{\partial}{\partial {x^n}} \right|_p$. The remaining part of the proof is now to show that the vectors $\left. \frac{\partial}{\partial {x^1}} \right|_p ,\ldots ,\left. \frac{\partial}{\partial {x^n}} \right|_p$ are linearly independent, i.e. they form a basis for the tangent space \tpm.\\
Suppose that $a_1,\ldots ,a_n\in\mathbb{R}^n$ are such that 
\[ \sum_{i=1}^{n} {a_i \left. \frac{\partial}{\partial {x^i}}\right|_p} =0.\]
Then
\[ 0= \sum_{i=1}^{n} {a_i \left. \frac{\partial {x^j}}{\partial {x^i}}\right|_p} =\sum_{i=1}^{n} {a_i \delta_{ij}} = a_j \]
for $j=1\ldots ,n$. Thus all integers $a_j$ for $j=1\ldots ,n$ must be zero, i.e. the vectors are linearly independent.\quad $\blacksquare$
\end{proof}

The following picture probably gives a better understanding of a globalization of a function on a differentiable manifold.

\begin{figure}[H]
\centering
\epsfig{file=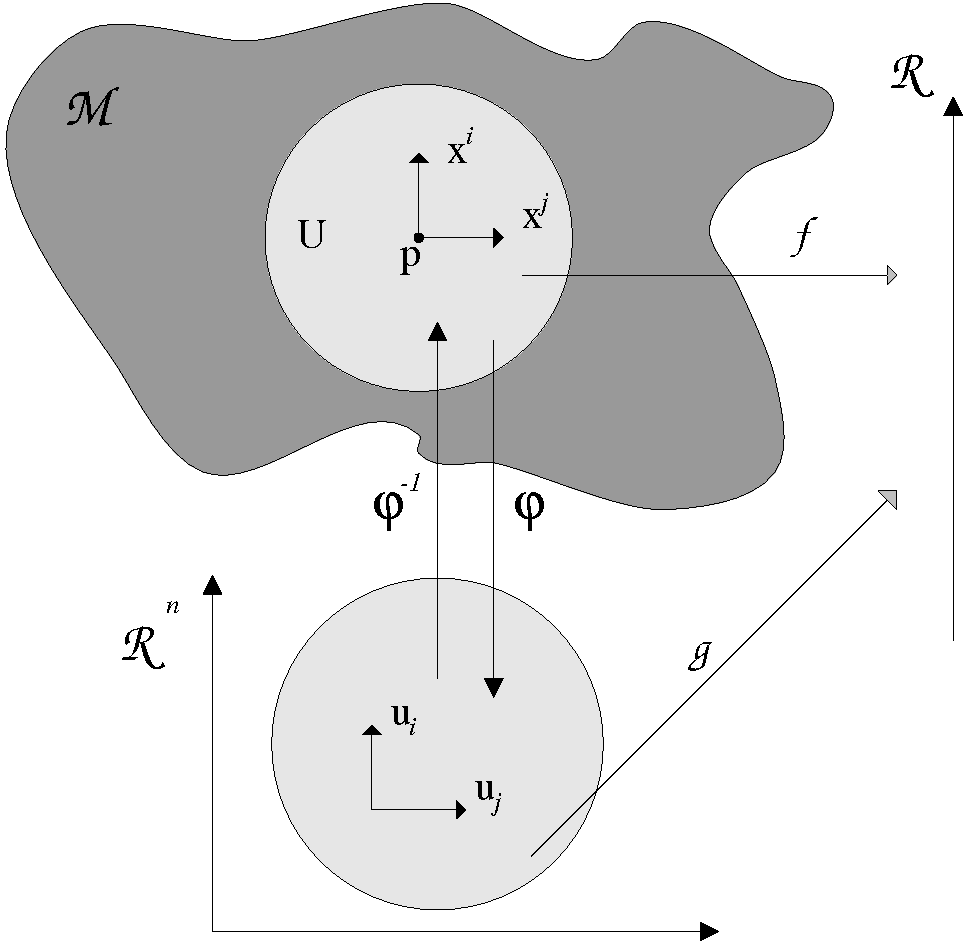, height=7cm, width=8cm}
\caption{The globalization $g$ of $f$.}
\end{figure}

\subsection{The Tangent Map}
In this section we show how a differentiable map $\Psi :\mathcal{M}\longrightarrow\mathcal{N}$ between differentiable manifolds \man{M} and \man{N} gives rise to a linear map between their tangent spaces which we will call the tangent map. The tangent map is the best linear approximation to a differentiable map between manifolds, which we start presenting a definition for. \\

\begin{defn}
Let $\Psi : \mathcal{M} \longrightarrow \mathcal{N}$ be a map between differentiable manifolds \man{M} and \man{N}. We say that $\Psi$ is differentiable if for any charts $(U,\varphi)$, $(V,\psi)$ on \man{M} and \man{N} respectively the mapping
\[ 
\psi \circ \Psi \circ \varphi^{-1} : \varphi (U) \longrightarrow \psi (V)
\]   
is $C^{\infty}$.
\end{defn}

\begin{defn}
Let $\Psi :\mathcal{M}\longrightarrow\mathcal{N}$ be a differentiable map between two differentiable manifolds \man{M} and \man{N}, and let $p\in \mathcal{M}$ and $\Psi (p) \in \mathcal{N}$. Then the tangent map $\Psi$ at $p$ denoted by $\Psi_{*p}$ is the map $\Psi_{*p} :\mathcal{T}_p \mathcal{M}\longrightarrow\mathcal{T}_{\Psi (p)} \mathcal{N}$ given by:

\begin{equation}
\Psi_{*p}(v_p)[f]=v_p [f\circ\Psi]
\end{equation}
for each $f\in\mathcal{F}(\mathcal{M})$ and $v_p \in \mathcal{T}_{p} \mathcal{M}$.
\end{defn}

\begin{theorem}
Let $\Psi :\mathcal{M}\longrightarrow\mathcal{N}$ be a differentiable map, and let $p\in \mathcal{M}$,\ $v_p \in \mathcal{T}_{p} \mathcal{M}$. Define $\Psi_{*p}(v_p) :\mathcal{F}(\mathcal{N}) \longrightarrow \mathbb{R}$ by $\Psi_{*p}(v_p)[f]=v_p [f\circ\Psi]$. Then
\[ \Psi_{*p}(v_p) \in \mathcal{T}_{\Psi (p)} \mathcal{N}. \]
\end{theorem}

\begin{proof}
Let $f,g\in\mathcal{F} (\mathcal{N}) \mbox{ and } \delta ,\rho \in \mathbb{R}$. Then

\begin{align*}
\Psi_{*p}(v_p)[fg] &=v_p [(fg)\circ \Psi] = v_p [(f\circ\Psi)(f\circ\Psi)]\\
 \\
&= (f\circ\Psi)(p)v_p [g\circ\Psi]+(g\circ\Psi)(p)v_p [f\circ\Psi]\\
 \\
&= f(\Psi (p)) \Psi_{*p}(v_p)[g]+g(\Psi (p)) \Psi_{*p}(v_p)[f].\\
\intertext{Thus $\Psi_{*p}(v_p)$ satisfies the Leibnitz property.}
\Psi_{*p}(v_p)[\delta f+\rho g] &=v_p [(\delta f+\rho g)\circ \Psi] = v_p[(\delta f\circ \Psi)+(\rho g\circ \Psi)]\\
 \\
&=\delta v_p[f\circ\Psi]+\rho v_p[g\circ\Psi]= \delta \Psi_{*p}(v_p)[f]+\rho \Psi_{*p}(v_p)[g].
\end{align*}
The Linear property is satisfied as well, and the proof is completed.    $\blacksquare$
\end{proof}

\begin{remark}
From the second part of the proof one can conclude that $\Psi_{*p} :\mathcal{T}_p \mathcal{M}\longrightarrow\mathcal{T}_{\Psi (p)} \mathcal{N}$ is a linear map between the vector spaces.
\end{remark}

\begin{theorem}
Let $(U_1,\varphi_1)$ be a chart on an $m$--dimensional manifold \man{M}, let $p$ denote a point in the chart domain $U_1 \subset \mathcal{M}$, and let $\varphi_{1} = (x^1,\ldots ,x^m)$ be the local coordinate system on $U_1$. Let $(V_1,\psi_1)$ be a chart on an $n$--dimensional manifold \man{N}, let $\Psi (p)$ denote a point in the chart domain $V_1 \subset \mathcal{N}$, and let $\psi_{1} = (y^1,\ldots ,y^n)$ be the local coordinate system on $V_{1}$. Assume that $\Psi : \mathcal{M} \longrightarrow \mathcal{N}$ is a differentiable map. Then for $i=1,\ldots ,m.$ we have
\[ \Psi_{*p}\left( \left. \frac{\partial}{\partial {x^i}}\right|_p \right) = \sum_{j=1}^{n} {\left. \frac{\partial {(y^j \circ\Psi)}}{\partial {x^i}}(p)\frac{\partial}{\partial {y^j}} \right|_{\Psi (p)}}.\]
\end{theorem}

\begin{proof}
Since $\Psi_{*p} \left( \left. \frac{\partial}{\partial {x^i}}\right|_p \right)\in\mathcal{T}_{\Psi (p)}\mathcal{N}$, and $\left( \left. \frac{\partial}{\partial {y^1}}\right|_{\Psi (p)}, \ldots ,\left. \frac{\partial}{\partial {y^n}}\right|_{\Psi (p)}\right) $ according to Theorem \ref{sec:tpmisofdimn} form a basis for $\mathcal{T}_{\Psi (p)}\mathcal{N}$, we may write

\begin{equation}
\label{eq:psi*}
\Psi_{*p} \left( \left. \frac{\partial}{\partial {x^i}}\right|_p \right) =\sum_{j=1}^{n} {a_{ij}\left. \frac{\partial}{\partial {y^j}}\right|_{\Psi (p)}}.
\end{equation}
We apply both sides to $y^k$, and we get:

\begin{align}
\label{eq:psi*1}
\Psi_{*p} \left( \left. \frac{\partial}{\partial {x^i}}\right|_p \right) [y^k] &=\left( \left. \frac{\partial}{\partial {x^i}}\right|_p \right) [y^k \circ \Psi ] = \left. \frac{\partial {(y^k \circ \Psi)}}{\partial {x^i}}\right|_p\\
a_{ik} &= \left( \sum_{j=1}^{n} {a_{ij}\left. \frac{\partial}{\partial {y^j}}\right|_{\Psi (p)}}\right) [y^k].
\end{align}
And from (\ref{eq:psi*}) and (\ref{eq:psi*1}) we get $a_{ik}= \left. \frac{\partial {(y^k \circ \Psi)}}{\partial {x^i}}\right|_p $. Similarly $a_{ij}= \left. \frac{\partial {(y^j \circ \Psi)}}{\partial {x^i}}\right|_p $ which gives us the possibility to rewrite (\ref{eq:psi*}) as 
\[ \Psi_{*p}\left( \left. \frac{\partial}{\partial {x^i}}\right|_p \right) = \sum_{j=1}^{n} {\left. \frac{\partial {(y^j \circ\Psi)}}{\partial {x^i}}(p)\frac{\partial}{\partial {y^j}} \right|_{\Psi (p)}}\]
which is the desired formula.\quad $\blacksquare$
\end{proof}

\clearpage{\pagestyle{empty}\cleardoublepage}

\section{The Projective Space}
Projective geometry is concerned with properties of incidence, that is properties which are invariant under stretching, translation or rotation of the plane. In the axiomatic development of the theory the notation of distance and angle will play no part. One of the most important examples of the theory is the Real Projective Plane. Here we give an introduction to the subject and a synthetic development gives an understanding to the Real Projective Space and the Complex Projective Space.

\subsection{The Real Projective Plane $\mathbb{R}{\mathbb{P}}^2$}

\begin{defn}
A set of points $P_{1}, \ldots ,P_{n}$ is said to be collinear if there exists a line $l$ containing them all.
\end{defn}

\begin{defn}
An affine plane is a set whose elements are called points and a set of subsets, called lines, satisfy the following three axioms.
\begin{description}
\item[A 1 ] Given two distinct points $P$, $Q$, there is one and only one line containing both $P$ and $Q$.
\item[A 2 ] Given a line $l$ and a point $P$, not on $l$, there is one and only one line $m$ which is parallel to $l$, and which passes through $P$.
\item[A 3 ] There exist three non-collinear points.
\end{description}
\end{defn}

\begin{example}
The ordinary plane, known from the Euclidean geometry, satisfies the axioms \textbf{A1--A3}, and therefore is an affine plane. A convenient way of representing this plane is by introducing Cartesian coordinates. Thus a point $P$ is represented as a pair $(x_1,x_2)$ of real numbers.
\end{example}

\begin{defn}
A relation $\sim$ is an equivalence relation if it has the following three properties.
\begin{description}
\item[1. Reflexive:] $a \sim a$
\item[2. Symmetric:] $a \sim b \Longrightarrow b \sim a$
\item[3. Transitive:] $a \sim b \; \wedge \; b \sim c \; \Longrightarrow \; a\sim c$
\end{description}
\end{defn}

\begin{example}
We say that two lines are parallel if they are equal, or if they have no points in common. Parallelism is an equivalence relation.\\
As a proof we check the three properties.  
\begin{enumerate}
\item Any line is parallel to itself, by definition.
\item $l\| m \Longrightarrow m\| l$, by definition.
\item If $l\| m$ and $m\| n$, we wish to prove that $l\| n$. Suppose $l$ is not parallel to $n$ and there is a point $P$ on the intersection of $l$ and $n$, i.e. $(P\in l\cap n)$. Then $l$ and $n$ are both parallel to $m$ and pass through $P$ which is impossible by axiom \textbf{A2}. We conclude that $l\cap n=\emptyset$, so $l||n$.
\end{enumerate}
\end{example}

\begin{defn}
A pencil of lines is either the set of all lines passing through some point $P$, or the set of all lines parallel to some line $l$.
\end{defn}

\begin{defn}
Let $A$ be an affine plane. For each line $l\in A$ we will call the pencil of lines parallel to $l$, an ideal point and denote it by $P^*$.
\end{defn} 

\begin{defn}
$S$ is a completion of $A$ if the points of $S$ are the points of $A$ plus all the ideal points of $A$. A line in $S$ is either:
\begin{itemize}
\item An ordinary point $l$ of $A$ plus the ideal point $P^*$ of $l$.
\item The line at infinity, consisting of all idea points of $A$.
\end{itemize}
\end{defn}

\begin{defn}
A projective plane ${\mathbb{P}}^2$ is a set whose elements are called points and a set of subsets, called lines, satisfy the following four properties.
\begin{description}
\item[P1.] Two distinct points $P$ and $Q$ of ${\mathbb{P}}^2$ lie on one and only one line.
\item[P2.] Any two lines meet in at least one point.
\item[P3.] There exist three non-collinear points.
\item[P4.] Every line contains at least three points.
\end{description}
\end{defn}

\subsubsection{Homogeneous coordinates in $\mathbb{P}^2$}
An easy way to introduce the homogeneous coordinates is to start with another construction of the real projective plane then the one we have done earlier.\\
Let $\mathbb{R}^3$ be the ordinary Euclidean 3--space, and let $O$ be a point of $\mathbb{R}^3$. Let $L$ be the set of lines through $O$. Define a point of $L$ to be a line through $O$ in $\mathbb{R}^3$ and define a line in $L$ to be the collection of lines through $O$ which all lie in the same plane through $O$. Then $L$ satisfies the properties \textbf{P1--P4} and so it is a projective plane. Now we are ready to introduce the homogeneous coordinates.\\
A point of ${\mathbb{P}}^2$ is a line $l$ through $O\in \mathbb{R}^3$. We will represent the point $P$ of $\mathbb{P}^2$ corresponding to $l$ by choosing any point $(x_1,x_2,x_3)\neq (0,0,0)$ on $l$. The numbers $(x_1,x_2,x_3)$ are the homogeneous coordinates of $P$.
\begin{figure}[htb]
\centering
\epsfig{file=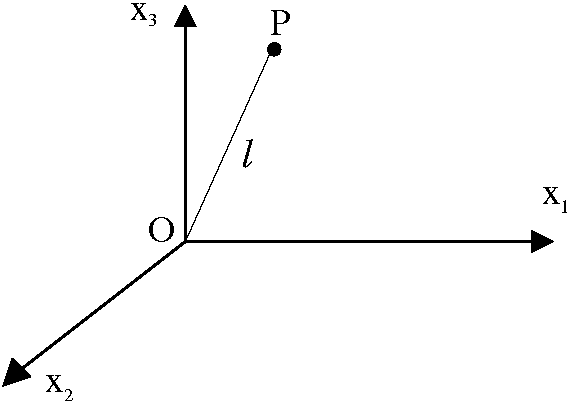, height=5cm}
\caption{The homogeneous coordinates $(x_1,x_2,x_3)$ of the point $P$.}
\label{fig:hcoord}
\end{figure}

Any other point of $l$ has the coordinates$(\lambda x_1, \lambda x_2, \lambda x_3)$ where $\lambda \in \mathbb{R} \setminus \{0\}$. Thus ${\mathbb{P}}^2$ is the collection of triples $(x_1, x_2, x_3)$ of real numbers, not all zero, and two triples $(x_1, x_2, x_3)$ and $(x'_1, x'_2, x'_3)$ represent the same point if and only if there exist $\lambda \in \mathbb{R}$ such that $x_i =\lambda x'_i$ for $i=1, 2, 3$. Since the equation of a plane in ${\mathbb{R}}^3$ passing through $O$ is of the form $\sum_{i=1}^{3}{a_i x_i}=0$ for $(a_1, a_2, a_3)\neq (0,0,0)$, we see that this is also the equation of a line of ${\mathbb{P}}^2$ in terms of the homogeneous coordinates.

\subsubsection{Topological view}
One can even look at the real projective plane from a topological point of view. There are useful topological descriptions of some elementary surfaces that are obtained from identifying edges of a square. For example, if we identify the top and the bottom edges of a square we obtain a cylinder. We describe this identification by means of a square with an arrow along the top edge and an arrow pointing in the same direction along the bottom edge. Now consider a square with the top and bottom edges identified, but in reverse order. This means that we twist one of the edges by $\pi$ before pasting them together. The resulting surface is the M\"{o}bius strip.\\
The surface that results when we identify the edges of the square with both the two vertical arrows pointing in different directions and the two horizontal arrows pointing in different directions is the real projective plane.
\begin{figure}[htb]
\centering
\subfigure[The Cylinder]{\epsfig{file=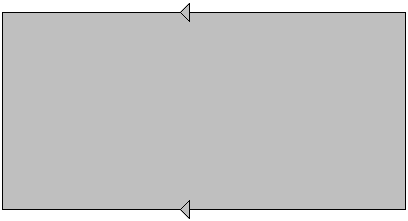, height=2cm}}
\subfigure[The M\"{o}bius Strip]{\epsfig{file=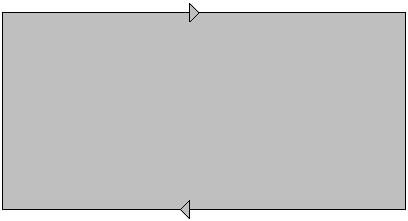, height=2cm}}
\subfigure[The Real Projective Plane]{\epsfig{file=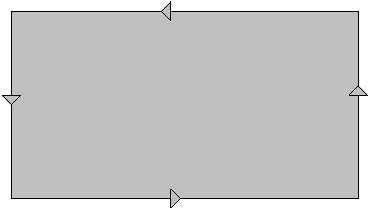, height=2cm}}
\label{fig:planes}
\end{figure}

One very important property of the M\"{o}bius strip and $\mathbb{R{P}}^2$ is that they both are nonorientable. One property of an n--dimensional nonorientable surface is that it can not be embedded in $\mathbb{R}^n$. To give a better understanding of this notion we first define an orientation of a manifold.

\begin{defn}
Let $\mathcal{M}$ be a connected m-dimensional differentiable manifold. At a point $p\in \mathcal{M}$, the tangent-space $\mathcal{T}_p\mathcal{M}$ is spanned by the basis $\{e_{\alpha} \}=\{{\partial}/{\partial x^{\alpha}} \}$, where $x^{\alpha}$ is the the local coordinate on the chart $U_i$ to which $p$ belongs. Let $U_j$ be another chart such that $U_i \cap U_j \neq \emptyset$, with the local coordinates $y^{\beta}$. If $p \in U_i \cap U_j $, then $\mathcal{T}_p\mathcal{M}$ is spanned by either $\{e_{\alpha}\}$ or $\tilde{e}_{\beta}= \{{\partial}/{\partial y^{\beta}} \}$. The basis changes as $\tilde{e}_{\beta}=({\partial x^{\alpha}}/{\partial y^{\beta}})e_{\alpha}$. If $det({\partial x^{\alpha}}/{\partial y^{\beta}})>0$ on $U_i \cap U_j$, $\{e_{\alpha}\}$ and $\{\tilde{e}_{\beta}\}$ are said to define the same orientation on $U_i \cap U_j$ and if $det({\partial x^{\alpha}}/{\partial y^{\beta}})<0$, they define the opposite orientation.
\end{defn}

\begin{defn}
Let $\mathcal{M}$ be a differentiable manifold. $\mathcal{M}$ is orientable if for any overlapping charts $U_i$ and $U_j$, there exist local coordinates $\{x^{\alpha}\}$ for $U_i$ and $\{y^{\beta}\}$ for $U_j$ such that $det({\partial x^{\alpha}}/{\partial y^{\beta}})>0$. Otherwise $\mathcal{M}$ is nonorientable 
\end{defn}

\begin{example}
The M\"{o}bius strip which is obtained as we described earlier, is a nonorientable surface. See figure \ref{fig:mobiusstrip}.\\
As we walk along the strip the coordinates changes, $x^1=y^1$ and $x^2=-y^2$, thus the determinant is $-1$.

\begin{figure}[H]
\centering
\epsfig{file=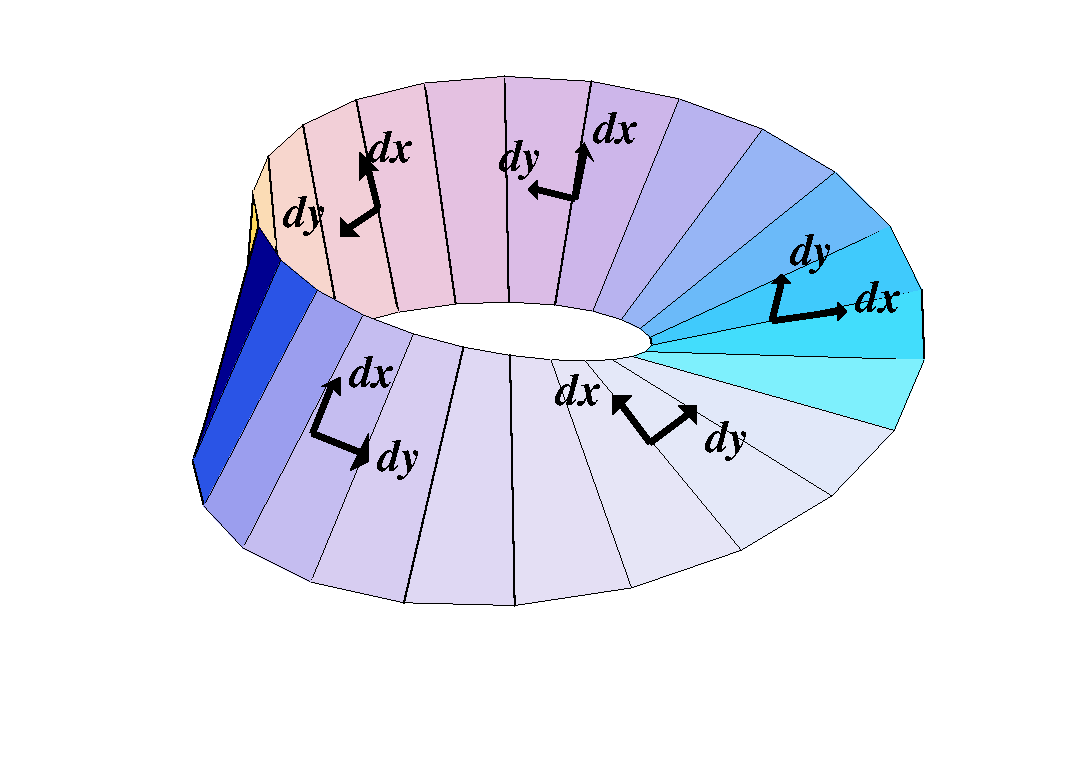, height=6cm, width=12cm}
\caption{The orientation of the coordinate system changes as we walk along the M\"{o}bius strip.}
\label{fig:mobiusstrip} 
\end{figure}
\end{example}

\subsubsection{Realization of $\mathbb{R{P}}^2$}
\begin{defn}
The pair of points $(p, -p)$ is called the pair of antipodal points.
\end{defn}

\begin{defn}
A map $f:{\mathbb{R}}^3\longrightarrow {\mathbb{R}}^3$ such that
\[ f(p)=f(-p) \]
is said to have the antipodal property.
\end{defn}

The real projective plane can even be thought of as a sphere with antipodal points identified, thus we can realize $\mathbb{R{P}}^2$ as the image of $S^2$ under a map $f$ which has the antipodal property.

\begin{example}
Steiner's Roman map $f_r (x,y,z)=(xy,yz,xz)$ has the antipodal property, and thus it realizes the real projective plane.\\
It is obvious that this map has the antipodal property, hence the map induces a map of $\mathbb{R{P}}^2$ onto $f_r(S^2(a))$. We call $f_r(S^2(a))$ Steiner's Roman surface of radius $a$. We can plot a portion of $f_r(S^2(a))$ by composing it with any patch on $S^2(a)$. Let us use the standard parameterization of the sphere, defined by:
\[ (u,v)\longmapsto (a \cos v \cos u, a\cos v\sin u, a\sin v) \]
Then the composition parameterizes all of Steiner's Roman surface of radius $a$. We get:
\[ \left( \frac{a^2\cos^{2} v\sin {2u}}{2}, \frac{a^2\sin u\sin {2v}}{2}, \frac {a^2\cos u\sin {2v}}{2}\right) \]
We plot this parameterized form of $f_r(S^2(a))$ and look at it from some different point of views:\\

\begin{figure}[H]
\centering
\subfigure[Back view.]{\epsfig{file=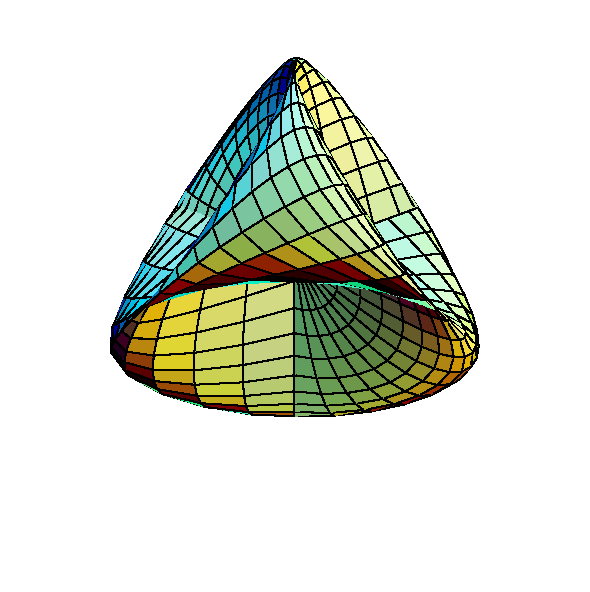, height=6cm}}
\subfigure[Front view.]{\epsfig{file=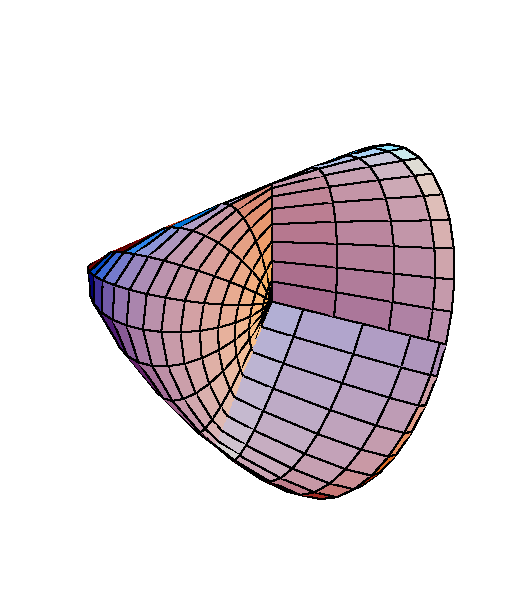, height=6cm}}
\subfigure[A front view with some part cut of.]{\epsfig{file=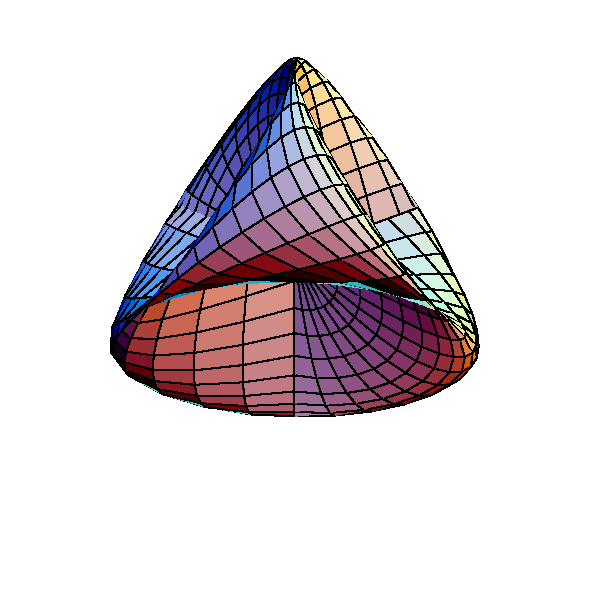, height=5cm}}
\subfigure[A side view with the same part cut of as in (c).]{\epsfig{file=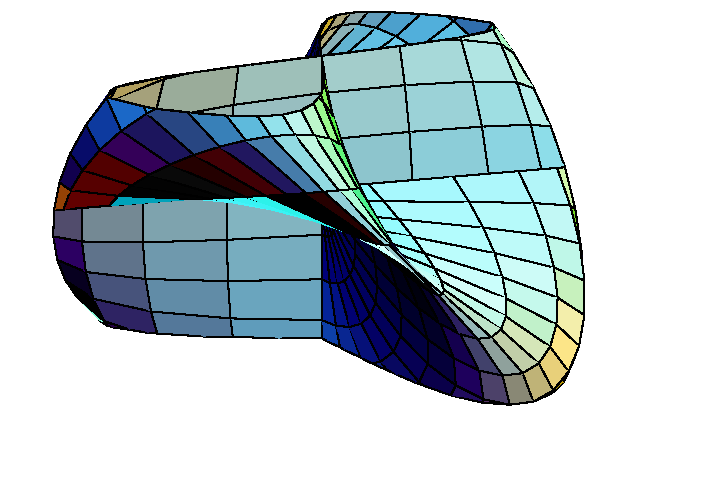, height=4cm}}
\caption{Steiner's Roman surface.}
\label{fig:RomanMap}
\end{figure}
\end{example}

\begin{example}
Another map called the Cross Cap with the antipodal property is given by
\[ f_c(x,y,z)= (yz, 2xy, x^2 - y^2) \]
We get a parameterization of the Cross Cap as in previous example by the standard parameterization of the sphere, then the composition gives the following parameterized form of $f_c(S^2(a))$
\[ \left( \frac{1}{2} a^2\sin u\sin {2v}, a^2\cos^{2} v\sin {2u}, a^2\cos {2u}\cos^{2} v\right) \]
This we can plot:
\begin{figure}[H]
\centering
\subfigure[One view of the Cross Cap.]{\epsfig{file=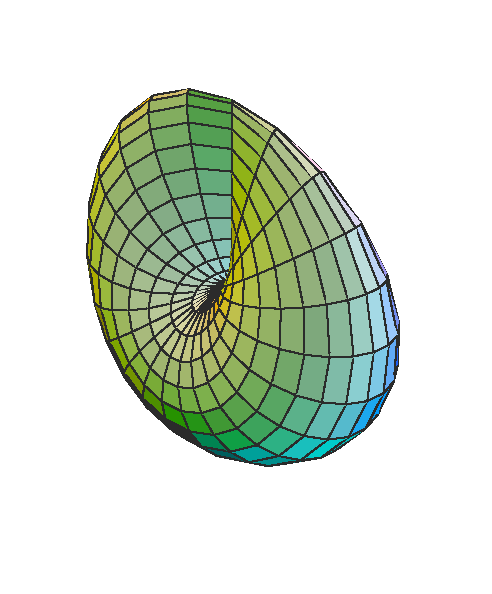, height=8cm}}
\subfigure[Another view of the Cross Cap. Its top is cut of.]{\epsfig{file=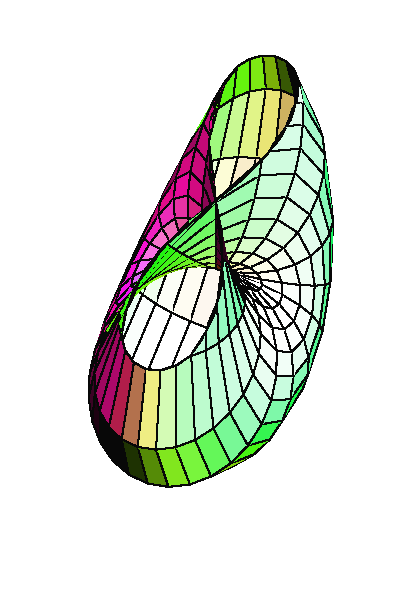, height=8cm}}
\caption{The Cross Cap.}
\label{fig:CrossCap}
\end{figure}
\end{example}

\subsection{The Real Projective Space $\mathbb{R{P}}^n$}
The real projective space denoted by $\mathbb{R{P}}^n$ is as we will see defined similarly to $\mathbb{R{P}}^2$, the difference is that the dimension of $\mathbb{R{P}}^n$ is usually greater than $2$.

\begin{defn}
The real projective space is the set of lines through the origin in $\mathbb{R}^{n+1}$. If $x=(x^0,\ldots ,x^n)\neq 0$, $x$ generates a line through the origin. Note that $y\in\mathbb{R}^{n+1}$ defines the same line as $x$ if there exists a real number $\delta\in\mathbb{R}^{n+1}\setminus\{0\}$ such that $y=\delta x$. Introduce an equivalence relation $\sim$ by $x\sim y$ if there exists a $\delta\in\mathbb{R}^{n+1}\setminus\{0\}$ such that $y=\delta x$. Then we introduce the following expression:
\[ \mathbb{R{P}}^n =( \mathbb{R}^{n+1}\setminus\{0\}) /\sim. \]
\end{defn}

\begin{defn}
The numbers $x^0,\ldots ,x^n$ are called the homogeneous coordinates of the equivalence class $[x]$ of $x=(x^{0}, \ldots ,x^{n})$ in $\mathbb{RP}^{n}$.
\end{defn}

Let $\pi : \mathbb{R}^{n+1} \setminus \{0\} \longrightarrow \mathbb{RP}^{n}$ be the natural projection. Thus 
\[
\pi (x)=[x]=\mbox{ ( the line through $0$ and $x$ in $\mathbb{R}^{n+1}$ ), $x\in \mathbb{R}^{n+1} \setminus \{ 0\} $.}
\]
We endow $\mathbb{RP}^{n}$ with the following topology. A subset $U\subset \mathbb{RP}^{n}$ is open if and only if $\pi^{-1}(U)$ is open in $\mathbb{R}^{n+1} \setminus \{ 0\}$. It can be shown that with this topology $\mathbb{RP}^{n}$ is compact. $\mathbb{R{P}}^n$ is in fact an $n$--dimensional manifold, an $(n+1)$--dimensional space with one--dimensional degree of freedom killed, thus the homogeneous coordinates $x^0,\ldots ,x^n$ that consists of $n+1$ elements can not be a good coordinate system. So we introduce the inhomogeneous coordinate system which is more useful than the homogeneous coordinate system.

\begin{defn}
Take the coordinate neighborhood $U_i$ as the set of lines with $x^i\neq 0$, that is $U_{i}=\{ [x^{0}, \ldots ,x^{n}]:x^{i}\neq 0\}$. Then we can introduce the inhomogeneous coordinates $\xi ^j_{(i)}$ on $U_i$ by
\[ 
\xi ^j_{(i)} = \frac{x^j}{x^i} \quad , \quad j \neq i. 
\]
\end{defn}

The inhomogeneous coordinates $\xi_{(i)}=(\xi^0_{(i)}, \ldots, \xi^{i-1}_{(i)}, \xi^{i+1}_{(i)}, \ldots ,\xi^n_{(i)})$ are well defined on $U_i$ since $x^i \neq 0$, and furthermore they are independent of the choice of the representative of the equivalence class since
\[ 
\frac{y^j}{y^i}=\frac{\delta x^j}{\delta x^i}=\frac{x^j}{x^i}.
\]
$\xi_{(i)}$ gives the coordinate map $\varphi_i : U_i \longrightarrow \mathbb{R}^n$, i.e. 
\[ 
\varphi_i :[x^0,\ldots ,x^n] \longmapsto (x^0/x^i, \ldots , x^{i-1}/x^i, x^{i+1}/x^i, \ldots , x^n/x^i). 
\]

\begin{defn}
For $x=(x^0, \ldots, x^n)\in U_i\cap U_j$we assign two inhomogeneous coordinates $\xi^k_{(i)}=x^k/x^i$ and $\xi^k_{(j)}=x^k/x^j$. Then the coordinate transformation $\Psi_{ij}=\varphi_i \circ \varphi_j^{-1}$ is 
\[ 
\Psi_{ij} : \xi^k_{(j)}\longmapsto \xi^k_{(i)} =(x^j/x^i)\xi^k_{(j)}. 
\]
Thus $\Psi_{ij}$ is nothing but multiplication by $(x^j/x^i)$.
\end{defn}

Consequently, the family $\{ (U_{i},\varphi_{i}) : i=0,\ldots ,n \}$ forms a $C^{\infty}$--atlas on $\mathbb{RP}^{n}$, and so $\mathbb{RP}^{n}$ is an $n$--dimensional $C^{\infty}$--manifold. \\ 
Just in the same way that we could think of $\mathbb{R{P}}^2$ as the sphere with the antipodal points identified, we can define $\mathbb{R{P}}^n$ as the unit sphere $S^{n} \subset \mathbb{R}^{n+1} $ with antipodal points identified. Here comes a description:\\
As a representative of the equivalence class, we may take points $|x|=1$ on a line through the origin. These are points on the unit sphere. Since there are two points on the intersection of a line with $S^n$ we have to take one of them consistently, i.e. nearby lines are represented by nearby points in $S^n$. This amounts to taking the hemisphere. Note that the antipodal points on the boundary are identified by the definition, $(x^0,\ldots ,x^n)\sim - (x^0,\ldots ,x^n)$. This hemisphere is homeomorphic to the unit ball $\mathcal{B}^{n}$ in \real{n} with antipodal points on the boundary of $\mathcal{B}^{n}$ identified.

\subsection{The Complex Projective Space $\mathbb{C{P}}^n$}
\begin{defn}
$z=(z^0, \ldots ,z^n)\in \mathbb{C}^{n+1}$ determines a complex line through the origin if $z\neq 0$. Define an equivalence relation $\sim$ by $z\sim w$ if there exists a complex number $\lambda\neq 0$ such that $w=\lambda z$. Then the complex projective space can be expressed as follows:
\[ \mathbb{C{P}}^n=(\mathbb{C}^{n+1}\setminus \{0\})/\sim. \]
\end{defn}
Similarly to $\mathbb{R{P}}^n$, the $n+1$ numbers $(z^0, \ldots ,z^n)$ are called the homogeneous coordinates of $[z]$. The topology on $\mathbb{CP}^{n}$ and the atlas are introduced as in $\mathbb{RP}^{n}$ with $\mathbb{C}$ replacing $\mathbb{R}$. A chart $U_\mu$ is a subset of $\mathbb{C}^{n+1}\setminus \{0\}$ such that $z^{\mu}\neq 0$. In the chart $U_\mu$, the inhomogeneous coordinates are defined by $\xi^\nu_{(\mu)}= z^\nu /z^\mu$, for $\nu\neq\mu$. In $U_\mu\cap U_\nu\neq\emptyset$, the coordinate transformation $\Psi_{\mu\nu} : \mathbb{C}^n\longrightarrow\mathbb{C}^n$ is
\[ \Psi_{\mu\nu} :\xi^\kappa_{(\nu)} \longmapsto \xi^\kappa_{(\mu)} = (z^\nu / z^\mu)\xi^\kappa_{(\nu)}.\]
Accordingly $\Psi_{\mu\nu}$ is a multiplication by $(z^\nu / z^\mu)$, which is holomorphic. In other words $\mathbb{CP}^{n}$ is a complex manifold of dimension $n$.

\begin{defn}
The Grassmannian manifold $G_{k,n}(\mathbb{C})$ is the set of $k$--dimensional surfaces of $\mathbb{C}^n$.
\end{defn}

\begin{defn}
If we delete $(n-k)$ rows and $(n-k)$ columns from an $(n\times n)$ matrix, then the remaining elements form a $(k\times k)$ matrix. The determinant of the remaining $(k\times k)$ matrix is called the $k\times k$ minor of the $(n\times n)$ matrix. Sometimes one even calls just the $(k\times k)$ matrix the $k\times k$ minor of the $(n\times n)$ matrix. This we will do in the following theorem. Moreover it might be useful to remark that an $(n\times n)$ matrix has $\binom{n}{k}^2$, $k\times k$ minors.
\end{defn}

\begin{theorem}
The Grassmannian manifold $G_{k,n}(\mathbb{C})$ is indeed a complex manifold.
\end{theorem}

\begin{proof}
We will sketch here the construction of a complex atlas on $G_{k,n}(\mathbb{C})$. Let $M_{k,n}(\mathbb{C})$ be the set of $(k\times k)$ matrices of rank $k$, $(k\leq n)$. Take $A=(a_{ij})\in M_{k,n}(\mathbb{C})$ and define $k$ vectors $a_i \mbox{ }(1\leq i\leq k)$ in $\mathbb{C}^n$ by $a_i=(a_{ij})$. Since $rank\mbox{ }A=k$, the $k$ vectors $a_i$ are linearly independent and span a $k$--dimensional plane in $\mathbb{C}^n$.\\
Let $GL(k,\mathbb{C})$ be the group of all nonsingular linear transformations of the $k$--dimensional complex vector space. Take $g\in GL(k,\mathbb{C})$ and consider a matrix $ \bar{A} =gA\in M_{k,n}(\mathbb{C})$, then $\bar{A}$ defines the same $k$--planes as $A$.\\
Introduce an equivalence relation $\sim$ by $\bar{A}\sim A$ if there exists $g\in GL(k,\mathbb{C})$ such that $\bar{A}=gA$.\\
We identify $G_{k,n}(\mathbb{C})$ with the coset space $M_{k,n}(\mathbb{C})/GL(k,\mathbb{C})$. Take $A\in M_{k,n}(\mathbb{C})$ and let $(A_1,\ldots ,A_l)$ be the collection of all $k\times k$ minors of $A$. Then $l=\binom{n}{k}$ because it is equal to the number of $k\times k$ minors of an $k\times n$ matrix. Since $rank \mbox{ } A=k$, there exists some $A_\alpha\mbox{ } (1\leq \alpha \leq l)$ such that $det\mbox{ } A\neq 0$.\\
Let us assume that the minor $A_1$ made of the first $k$ columns has non-vanishing determinant. Then $A=(A_1,\tilde{A}_1)$ where $\tilde{A}_1$ is a $k\times (n-k)$ matrix. Then $A_1^{-1}\cdot A = (\mathbb{I}_k,A_1^{-1}\cdot \tilde{A}_1)$, Where $\mathbb{I}_k$ is the $k\times k$ unit matrix. Note that $A_1^{-1}$ always exists since $det\mbox{ } A_1\neq 0$.\\
Thus the degrees of freedom are given by the entries of the $k\times (n-k)$ matrix $A_1^{-1}\cdot \tilde{A}_1$. We denote this subset of $G_{k,n}(\mathbb{C})$ by $U_1$, where $U_1$ is a coordinate neighborhood whose coordinates are given by $k\cdot (n-k)$ entries of $A_1^{-1}\cdot \tilde{A}_1$. In the case that $det\mbox{ } A_\alpha\neq 0$, where $A_\alpha$ is composed of the columns $(i_1,\ldots ,i_k)$, we multiply $A_1^{-1}$ to obtain the representative of the set to which $A$ belongs to be. \\
Similarly to the case with $A_1$, but written in matrix form we get:
\begin{equation}
\label{eq:matris} 
A^{-1}_{\alpha}\cdot A =
\left(
\begin{array}{ccccccc}
\mathbf{Column}     & \mathbf{i_1}    & \ldots & \mathbf{i_2} & \ldots & \mathbf{i_k} & \ldots \\         
\ldots & 1          & \ldots & 0        & \ldots & 0            & \ldots \\
\ldots & 0          & \ldots & 1        & \ldots & 0            & \ldots \\
\ldots & \vdots     & \ldots & \vdots   & \ldots & \vdots       & \ldots \\
\ldots & 0          & \ldots & 0        & \ldots & 1            & \ldots \\
\end{array}
\right)
\end{equation}
where the entries not written explicitly form a $k\times(n-k)$ matrix. We denote this subset of $M_{k,n}(\mathbb{C})$ with $det\mbox{ }A_\alpha\neq 0$ by $U_\alpha$. Now we have defined the chart $U_\alpha$ to be a subset of $G_{k,n}(\mathbb{C})$ such that $det\mbox{ }A_\alpha\neq 0$. The $k\cdot (n-k)$ coordinates on $U_\alpha$ are given by the entries of the $k\times (n-k)$ matrix $A^{-1}_\alpha A$.    $\blacksquare$
\end{proof} 

The relation between the Projective space and the Grassmannian manifold is now evident. An element of $M_{1,n+1}(\mathbb{C})$ is a vector $A=(x^0,\ldots ,x^n)$. Since the $\alpha$--th minor $A_\alpha$ of $A$ is a number $x^\alpha$, the condition $det\mbox{ }A_\alpha\neq 0$ becomes $x^\alpha\neq 0$. The representation (\ref{eq:matris}) is just the inhomogeneous coordinates
\[ 
{(x^\alpha )}^{-1}\cdot (x^0,\ldots ,x^\alpha ,\ldots ,x^n)=(x^0 / x^{\alpha} , \ldots ,x^{\alpha -1}/x^{\alpha} ,x^{\alpha +1}/x^{\alpha} ,\ldots ,x^n/x^{\alpha} ). 
\]

\begin{corollary}
$G_{1,n+1}(\mathbb{C})=\mathbb{CP}^{n}$.
\end{corollary}

\clearpage{\pagestyle{empty}\cleardoublepage}

\section{The Schwarzian Curvature}
In this section we start with constructing the moving frame on curves in the complex projective space, in terms of their liftings. Then we define the Schwarzian curvatures denoted by $\kappa_i$ and give the formulas for the $\kappa$'s for curves in $\mathbb{CP}^n$. We finish with some ``Low-dimension'' examples and transformation rules for the change of coordinates.\\

Let $\Phi : \mathcal{D} \longrightarrow\mathbb{CP}^{n} $ be an analytic curve in the $n$--dimensional complex projective space, where \man{D} is the unit disc in $\mathbb{C}$. We will assume that $\Phi$ can be lifted to a holomorphic curve $f : \mathcal{D} \longrightarrow \mathbb{C}^{n+1} \setminus \{ 0 \}$ such that $f, f', \ldots ,f^{(n)}$ form a linearly independent set of vector valued functions. Two such liftings $f_1$ and $f_2$ are equivalent if and only if $f_1 = \lambda f_2$, where $\lambda$ is an analytic nonzero function. \\
Now we construct the moving frame on the curve $f$ as follows.\\
As the first vector we take the vector valued function $\nu = \lambda f$, $\lambda\neq 0$ and the remaining vectors can be obtained from $\nu$. We denote them by $e_1, \ldots ,e_n$.
\[
\left\{ \begin{array}{l}
e_1={\nu}',\\ e_2=e'_{1}, \\ \ldots  \\ e_n=e'_{n-1}.
\end{array} \right.
\]
Then
\begin{align}
\label{eq:det1}
det \left[ \begin{array}{cccc}
\uparrow     & \uparrow    &        & \uparrow \\
\nu          & e_1         & \ldots & e_n      \\
\downarrow   &  \downarrow &        &\downarrow 
\end{array} \right]        &= det \left[ \begin{array}{ccc}
\uparrow     & \uparrow    &  \\
\lambda f    & (\lambda' f +\lambda f')        & \ldots\\
\downarrow   &  \downarrow &          
\end{array} \right].
\end{align}
The general formula for the derivatives is a suitable formula to look at. The reason is that
\[
(\lambda f)^{(m)}=\sum_{j=0}^{m}{\binom{m}{j}\lambda^{(j)}f^{(m-j)}}
\]
tells us how the derivatives of $\lambda f$ look like. We see that in the expression of the $m$--th derivative of $\lambda f$, there is only one ``element'' that is not repeated in the previous columns, namely $\lambda f^{(m)}$. Thus the $(m+1)$st column of (\ref{eq:det1}) is 

\[
\left[ \begin{array}{c}
\uparrow \\
(\lambda f)^{(m)}=\sum_{j=0}^{m}{\binom{m}{j}\lambda^{(j)}f^{(m-j)}} \\
\downarrow 
\end{array} \right] = \left[ \begin{array}{c}
\uparrow \\
\lambda f^{(m)} \\
\downarrow
\end{array} \right] + \ldots + \left[ \begin{array}{c}
\uparrow \\
\lambda^{(m)} f \\
\downarrow
\end{array} \right],
\]
and since two linearly dependent columns, or rows, make the determinant to vanish, we get
\[ 
det \, \left[ \begin{array}{ccc}
        &\uparrow             &              \\
\ldots  & (\lambda f)^{(m)}   & \ldots       \\
        &\downarrow           &   
\end{array} \right]= det \, \left[ \begin{array}{ccc}
       & \uparrow             &         \\
\ldots & \lambda f^{(m)}      & \ldots  \\
       & \downarrow           &       
\end{array} \right].
\]
Note that we only have written what happens in the $(m+1)$st column of (\ref{eq:det1}). Repeating this for the other columns in (\ref{eq:det1}) we obtain

\begin{align}
\label{eq:det3}
det \left[ \begin{array}{cccc}
\uparrow     & \uparrow    &        & \uparrow \\
\nu          & e_1         & \ldots & e_n      \\
\downarrow   &  \downarrow &        &\downarrow 
\end{array} \right] &= det \left[ \begin{array}{cccc}
\uparrow     & \uparrow    & & \uparrow \\
\lambda f          & \lambda f'        & \ldots  & \lambda f^{(n)}     \\
\downarrow   &  \downarrow & & \downarrow          
\end{array} \right] \nonumber\\
&= \lambda^{n+1} det \left[ \begin{array}{cccc}
\uparrow     & \uparrow    & & \uparrow \\
f          & f'        & \ldots  & f^{(n)}     \\
\downarrow   &  \downarrow & & \downarrow          
\end{array} \right].
\end{align}
We also know that we can choose $\lambda$ so that

\begin{equation}
\label{eq:lambdan+1}
\lambda^{n+1} det \left[ \begin{array}{cccc}
\uparrow     & \uparrow    & & \uparrow \\
f          & f'        & \ldots  & f^{(n)}     \\
\downarrow   &  \downarrow & & \downarrow          
\end{array} \right] =1, 
\end{equation}
since $\Phi$ is an analytic on \man{D}, then the components of $f$ must be analytic in the one--dimensional complex sense. According to this, the elements of the derivatives of $f$ also are analytic and since the determinant is nothing but a sum of products of the elements of $f$ and its derivatives up to order $n$, the determinant is analytic on \man{D} and  $\frac{1}{det [f, \ldots ,f^{(n)}]}$ is analytic as well since the determinant is never equal to zero. We want to show the existence of an analytic function $\lambda : \mathcal{D} \longrightarrow \mathbb{C}\setminus \{ 0 \}$ such that 

\[
\lambda^{n+1} = \frac{1}{det \left[ \begin{array}{cccc}
\uparrow     & \uparrow    &         & \uparrow \\
f            & f'          & \ldots  & f^{(n)}     \\
\downarrow   &  \downarrow &         & \downarrow          
\end{array} \right]}.
\]
Let $\varphi (z) = \frac{1}{det [f, \ldots ,f^{(n)}]}$, $z \in \mathcal{D}$. Then $\varphi : \mathcal{D} \longrightarrow \mathbb{C}\setminus \{ 0 \}$ is analytic and hence so is $\frac{\varphi'}{\varphi}$. Define $\psi : \mathcal{D} \longrightarrow \mathbb{C}$ by 

\[
\psi (z) = \int_{[0,z]} \frac{\varphi' (w)}{\varphi (w)}\, dw \; , \; z \in \mathcal{D}.
\]
Then $\psi$ is analytic and $\psi' = \frac{\varphi'}{\varphi}$. By adding a constant to $\psi$ ( if necessary ) we may suppose that
\begin{align}
\label{eq:exppsi}
e^{\psi (0)}          & = \varphi (0).
\intertext{Now}
(\varphi e^{-\psi})'  & = \varphi' e^{-\psi} - \psi' \varphi e^{-\psi} = 0 \nonumber
\end{align}
in \man{D}. Hence $\varphi e^{-\psi}$ is constant and thus $\varphi = e^{\psi}$ in \man{D} because of (\ref{eq:exppsi}). Now we put

\[
\lambda (z) = e^{(\frac{1}{n+1} \psi (z))}\; , \; z \in \mathcal{D}.
\] 
So with this choice of $\lambda$ we have shown that 

\begin{equation}
\label{eq:det4}
det \left[ \begin{array}{cccc}
\uparrow     & \uparrow    &        & \uparrow \\
\nu          & e_1         & \ldots & e_n      \\
\downarrow   &  \downarrow &        &\downarrow 
\end{array} \right] =1.
\end{equation}
Since the determinant in (\ref{eq:det4}) is different from zero, we conclude that the vectors $\nu, e_1,\ldots ,e_n$ are linearly independent, and thus we have constructed a frame for $f(z)$ which we call the canonical frame of $f(z)$.\\
Differentiating (\ref{eq:det4}), we get

\begin{align*}
det \left[ \begin{array}{cccc}
\uparrow     & \uparrow    &        & \uparrow \\
\nu'          & e_1         & \ldots & e_n      \\
\downarrow   &  \downarrow &        &\downarrow 
\end{array} \right] & + det \left[ \begin{array}{cccc}
\uparrow     & \uparrow    &        & \uparrow \\
\nu          & e'_1         & \ldots & e_n      \\
\downarrow   &  \downarrow &        &\downarrow 
\end{array} \right] \\
+ \ldots &+ det \left[ \begin{array}{ccccc}
\uparrow     & \uparrow    &        & \uparrow  & \uparrow \\
\nu          & e_1         & \ldots & e_{n-1}   &e'_n      \\
\downarrow   &  \downarrow &        &\downarrow &\downarrow 
\end{array} \right]=0,
\end{align*}
and since in all the terms of the left hand side, except the last one, we have linearly dependent columns ( since $\nu' = e_1$ etc. ) we get

\begin{equation}
\label{eq:det5}
 det \left[ \begin{array}{ccccc}
\uparrow     & \uparrow    &        & \uparrow  & \uparrow \\
\nu          & e_1         & \ldots & e_{n-1}   &e'_n      \\
\downarrow   &  \downarrow &        &\downarrow &\downarrow 
\end{array} \right]=0.
\end{equation}
Now we are ready to define what we call the Schwarzian curvatures of $\Phi$.

\begin{defn}
\label{def:schwarz}
Solving (\ref{eq:det5}), we obtain 
\begin{equation*}
e'_{n}=\kappa_{0}\nu + \kappa_{1} e_{1} +\ldots + \kappa_{n-1} e_{n-1}.
\end{equation*}
We call the $\kappa 's$ the Schwarzian curvatures of $\Phi$.
\end{defn}

We are even able to look at what we have done from another point of view. We can see our tangent vectors as the vectors in the Frenet Frame, and this gives us the Frenet equations
\[
\left\{ \begin{array}{l}
\nu'=e_1,\\ e'_1=e_2, \\ \ldots  \\ e'_{n-1}=e_n \\ e'_n=\kappa_{0}\nu + \kappa_{1} e_{1} +\ldots + \kappa_{n-1} e_{n-1}.
\end{array} \right.
\]
This we can write as
\begin{align*}
\left[ \begin{array}{ccc}
\uparrow     &         & \uparrow\\
\nu          &\ldots   &e_n\\
\downarrow   &         & \downarrow
\end{array} \right]'   &= \left[ \begin{array}{cccc}
\uparrow      &         & \uparrow      &\uparrow\\
e_1           &\ldots   &e_n            &( \kappa_{0}\nu + \ldots + \kappa_{n-1} e_n ) \\
\downarrow    &         & \downarrow    &\downarrow
\end{array} \right] \\
& =\left[ \begin{array}{ccc}
\uparrow     &         & \uparrow\\
\nu          &\ldots   &e_n\\
\downarrow   &         & \downarrow
\end{array} \right] \mbox{K},
\intertext{where}
\mbox{K} &=\left[ \begin{array}{cccccc}
0        & 0      & \ldots & 0      & 0      & \kappa_{0}\\
1        & 0      & \ldots & 0      & 0      & \kappa_{1}\\
0        & 1      & \ldots & 0      & 0      & \kappa_{2}\\
\vdots   & \vdots & \ddots & \vdots & \vdots & \vdots\\
0        & 0      & \ldots & 1      & 0      & \kappa_{n-1}\\
0        & 0      & \ldots & 0      & 1      & 0
\end{array} \right].
\end{align*}

\begin{theorem}
The quantities $\kappa_{0},\ldots ,\kappa_{n-1}$ are invariant under projective transformations in $\mathbb{CP}^{n}$, or equivalently, under affine non--singular transformations of $\mathbb{C}^{n+1}$.
\end{theorem}

\begin{proof}
Let $A :\mathbb{C}^{n+1} \longrightarrow \mathbb{C}^{n+1}$ be an affine nonsingular transformation. Let the transformed $f$ be denoted by $\tilde{f}$, i.e. $\tilde{f}= A\circ f$. Then, just in the same way as earlier, we get

\begin{align*}
1 = det \left[ \begin{array}{ccc}
\uparrow               &         & \uparrow\\
\tilde{\nu}            &\ldots   &\tilde{e}_{n}\\
\downarrow             &         & \downarrow
\end{array} \right]    & = \tilde{\lambda}^{n+1} det \left[ \begin{array}{ccc}
\uparrow               &         & \uparrow\\
(A\circ f)             &\ldots   &(A\circ f)^{(n)}\\
\downarrow             &         & \downarrow
\end{array} \right]\\
&=\tilde{\lambda}^{n+1} \; det \mathcal{J}_{A} \; det \left[ \begin{array}{ccc}
\uparrow        &         & \uparrow\\
f               &\ldots   &f^{(n)}\\
\downarrow      &         & \downarrow
\end{array} \right],
\end{align*}
where $\mathcal{J}_{A}$ is nothing but the the Jacobian matrix of A.\\
Comparing with (\ref{eq:det3}), we see that

\[
\lambda^{n+1} det \left[ \begin{array}{ccc}
\uparrow      &         & \uparrow\\
f             &\ldots   & f^{(n)}\\
\downarrow    &         & \downarrow
\end{array} \right] = \tilde{\lambda}^{n+1} \; det \mathcal{J}_{A} \; det \left[ 
\begin{array}{ccc}
\uparrow        &         & \uparrow\\
f               &\ldots   &f^{(n)}\\
\downarrow      &         & \downarrow
\end{array} \right],
\]
which gives us

\[
\lambda = \tilde{\lambda}\sqrt[n+1]{det \mathcal{J}_{A}} \quad \Longrightarrow \quad 
\tilde{\lambda}= \frac{\lambda}{\sqrt[n+1]{det \mathcal{J}_{A}}}.
\]
The affine transformation $A$ is a composition of a linear part and a constant part, thus $det \mathcal{J}_{A}$ is just the determinant of the linear part of $A$. Hence $det \mathcal{J}_{A} \neq 0$ because $A$ is non--singular. We have 

\[
\tilde{\nu}=\tilde{\lambda} \; (A\circ f) = \frac{\lambda}{\sqrt[n+1]{det \mathcal{J}_{A}}} A\circ f = \frac{A\circ {\lambda f}}{\sqrt[n+1]{det \mathcal{J}_{A}}} = \frac{A\circ \nu}{\sqrt[n+1]{det \mathcal{J}_{A}}},
\]
and

\[
\tilde{e}_{i} =\frac{A\circ {e_i}}{\sqrt[n+1]{det \mathcal{J}_{A}}}.
\]
Therefore

\begin{align*}
\left[ \begin{array}{ccc}
\uparrow             &         & \uparrow\\
\tilde{\nu}          &\ldots   &\tilde{e}_{n}\\
\downarrow           &         & \downarrow
\end{array} \right]' &= \frac{A}{\sqrt[n+1]{det \mathcal{J}_{A}}}
\left[ \begin{array}{ccc}
\uparrow     &         & \uparrow\\
\nu          &\ldots   & e_n\\
\downarrow   &         & \downarrow
\end{array} \right]'\\
&= \frac{A}{\sqrt[n+1]{det \mathcal{J}_{A}}}
\left[ \begin{array}{ccc}
\uparrow     &         & \uparrow\\
\nu          &\ldots   & e_n\\
\downarrow   &         & \downarrow
\end{array} \right]\mbox{K}\\
&=\left[ \begin{array}{ccc}
\uparrow             &         & \uparrow\\
\tilde{\nu}          &\ldots   &\tilde{e}_{n}\\
\downarrow           &         & \downarrow
\end{array} \right] \mbox{K}. \quad \blacksquare
\end{align*}
\end{proof}

In some way one can see the structure of $\Phi$ by looking at the Schwarzian curvatures. If all of the $\kappa$'s of a curve $\Phi$ vanish, then we have $e'_n = 0$. Consequently all components of $\lambda f$ are polynomials of degree less than or equal to $n$, since $e'_n$ is just the $(n+1)$st derivative of $\lambda f$.\\
Furthermore we know that $\Phi$ is in the projective space $\mathbb{CP}^n$ and thus we can give the inhomogeneous coordinates of $\Phi$ by polynomials of degree less than or equal to $n$. $\Phi$ is seen to be an $n$--th degree polynomial fractional map into $\mathbb{CP}^{n}$, i.e.
\[
\Phi = \left( \frac{P_{1}(z)}{P_{0}(z)}, \ldots ,\frac{P_{n}(z)}{P_{0}(z)} \right).
\]
If all the Schwarzian curvatures of $\Phi$ are constants, then $\nu = \lambda f$ satisfies

\begin{equation}
\label{eq:vectoreqn}
\nu^{n+1} = \kappa_{0}\nu + \ldots + \kappa_{n-1}\nu^{(n-1)},
\end{equation}
which is an $(n+1)$st order differential equation with constant coefficients. It has $n+1$ linearly independent solutions $\nu_{0}, \ldots , \nu_{n}$, consisted of exponential functions and polynomials. Each component of $\lambda f$ is a linear combination of the $\nu$'s, and the inhomogeneous coordinates of $\Phi$ consists of these linear combinations. Thus
\[
\Phi =  \left( \frac{L_{1}}{L_{0}}, \ldots ,\frac{L_{n}}{L_{0}} \right),
\]
where the $L$'s are the linear combinations of the $\nu$'s.\\
Of course, these are not the only values that can be obtained by the $\kappa$'s. The $\kappa$'s might as well be functions of $z$. This case we will deal with now.

\begin{theorem}
Let $ k_{0}(z), \ldots , k_{n-1}(z)$ be $n$ analytic functions on the unit disc $\mathcal{D}$. Then there are curves with $\kappa_{0} = k_{0},\ldots ,\kappa_{n-1} = k_{n-1}$.
\end{theorem}

\begin{proof}
Consider the following differential equation
\begin{equation}
\label{eq:diffeq}
y^{(n+1)} = k_{0}y + \ldots + k_{n-1}y^{n-1}.
\end{equation}
It has $n+1$ linearly independent solutions $y_{0}, \ldots ,y_{n}$.\\
Let the curve $\Phi$ have the following inhomogeneous coordinates
\[
\Phi =  \left( \frac{y_{1}}{y_{0}}, \ldots ,\frac{y_{n}}{y_{0}} \right).
\]
Then we can write the lifting of $\Phi$ as $f=(y_{0}, \ldots ,y_{n})$. This gives us the opportunity to rewrite (\ref{eq:diffeq}) as $ f^{(n+1)}=k_{0}f + \ldots + k_{n-1}f^{(n-1)}$.\\
Now recall the formula by which we defined the Schwarzian curvatures in definition \ref{def:schwarz}, which was obtained from solving (\ref{eq:det5}). Similarly we obtain the following equation:
\begin{align}
\label{eq:detprim}
det \left[ \begin{array}{ccc}
\uparrow        &         & \uparrow\\
f               &\ldots   &f^{(n)}\\
\downarrow      &         & \downarrow
\end{array} \right] '    &= det \left[ \begin{array}{ccccc}
\uparrow         &\uparrow      &           &\uparrow      & \uparrow\\
f                &f'            &\ldots     &f^{(n-1)}     &f^{(n+1)}\\
\downarrow       & \downarrow   &           & \downarrow   & \downarrow
\end{array} \right]=0
\intertext{which implies}
det \left[ \begin{array}{ccc}
\uparrow        &         & \uparrow\\
f               &\ldots   &f^{(n)}\\
\downarrow      &         & \downarrow
\end{array} \right]       &= \c{C} \nonumber
\end{align}
where $\c{C}$ is a constant. Thus in the canonical frame we can choose $\lambda$ to be the $(n+1)$st root of $\c{C}$. Under this choice (\ref{eq:detprim}) is equivalent to $e'_{n}= k_{0} \nu + \ldots + k_{n-1} e_{n-1}$, and we have the desired result.\quad $\blacksquare$
\end{proof}

\begin{remark}
Remark that the solutions to the system must be of the form $f = (y_{0}, \ldots, y_{n})$ where the $y$'s are linearly independent solutions of (\ref{eq:diffeq}). By the $y$'s we intend to give an understanding of how the $\nu$'s would behave.
\end{remark}

\subsection{Formulas for Schwarzian Curvatures}
One goal of these calculations is of course to give the formulas for the $\kappa$'s. For this purpose we start by recalling the vector equation (\ref{eq:vectoreqn}) and that we defined $e_{i}$ as the $i$--th derivative of $\nu$, $\nu^{(i)}$. Then we can write $e'_{n}= \kappa_{0} \nu + \ldots + \kappa_{n-1} e_{n-1}$, which can be written in matrix form.
\[
\left[ \begin{array}{ccc}
\uparrow          &         & \uparrow\\
\nu               &\ldots   &e_{n}\\
\downarrow        &         & \downarrow
\end{array} \right] \left[ 
\begin{array}{c} \kappa_{0} \\ \vdots \\ \kappa_{n-1} \\ 0 \end{array} \right] =
\left[ \begin{array}{c} 
\uparrow \\ 
e'_{n} \\ 
\downarrow  
\end{array} \right]
\]
We can apply Cramer's Rule for giving the expression for the $\kappa$'s. We get
\[
\kappa_{i} = \frac{det \left[ \begin{array}{ccccccc}
\uparrow     &        & \uparrow   & \uparrow   & \uparrow   &        & \uparrow\\
\nu          &\ldots  &e_{i-1}     &e'_{n}      &e_{i+1}     &\ldots  &e_{n}\\
\downarrow   &        & \downarrow & \downarrow & \downarrow &        & \downarrow
\end{array} \right]}{det \left[ \begin{array}{ccc}
\uparrow          &           & \uparrow\\
\nu               &\ldots     &e_{n}\\
\downarrow        &           & \downarrow
\end{array} \right]}.
\]
The determinant in the denominator is different from zero, from which we conclude that the matrix is nonsingular which furthermore is a requirement for the application of Cramer's Rule. In fact the determinant in the denominator is equal to one, hence

\[
\left\{ \begin{array}{l}
\kappa_{0} = det \left[ \begin{array}{cccc}
\uparrow    &\uparrow          &         & \uparrow\\
e'_{n}      &e_{1}             &\ldots   &e_{n}\\
\downarrow  &\downarrow        &         & \downarrow
\end{array} \right]            = (-1)^{n} \; det \left[ \begin{array}{cccc}
\uparrow    &                  & \uparrow          & \uparrow\\
e_{1}       &\ldots            & e_{n}             &e'_{n}\\
\downarrow  &                  & \downarrow        & \downarrow
\end{array} \right], \\
 \\
\kappa_{1} = det \left[ \begin{array}{ccccc}
\uparrow    &\uparrow          &\uparrow   &         & \uparrow\\
\nu         &e'_{n}            &e_{2}      &\ldots   &e_{n}\\
\downarrow  &\downarrow        &\downarrow &         & \downarrow
\end{array} \right]            = (-1)^{n-1} \; det \left[ \begin{array}{ccccc}
\uparrow    &\uparrow       &          & \uparrow          & \uparrow\\
\nu         &e_{2}          &\ldots    & e_{n}             &e'_{n}\\
\downarrow  &\downarrow     &          & \downarrow        & \downarrow
\end{array} \right], \\
 \\
\ldots \\
 \\
\kappa_{n-1} = det \left[ \begin{array}{ccccccc}
\uparrow    &\uparrow          &        &\uparrow           & \uparrow    & \uparrow\\
\nu         &e_{1}             &\ldots  &e_{n-2}            &e'_{n}       &e_{n}\\
\downarrow  &\downarrow        &        & \downarrow        & \downarrow  & \downarrow
\end{array} \right]            = - \; det \left[ \begin{array}{ccccc}
\uparrow    &          & \uparrow   & \uparrow          & \uparrow\\
\nu         &\ldots    & e_{n-2}    & e_{n}             &e'_{n}\\
\downarrow  &          & \downarrow & \downarrow        & \downarrow
\end{array} \right]. \\
\end{array} \right.
\]
Thus up to sign, the Schwarzian curvatures are determinants of $(n+1)\times (n+1)$ submatrices of the $(n+1)\times (n+2)$ matrix

\begin{align*}
\left[ \begin{array}{ccccc}
\uparrow    & \uparrow    &         & \uparrow  & \uparrow\\
\nu         & e_1         & \ldots  & e_n       & e'_n    \\
\downarrow  & \downarrow  &         &\downarrow & \downarrow
\end{array} \right]       &= \left[  \begin{array}{ccccc}
\uparrow    & \uparrow     &         & \uparrow        & \uparrow     \\
\nu         & \nu'         & \ldots  & \nu^{(n)}       & \nu^{(n+1)}  \\
\downarrow  & \downarrow   &         &\downarrow       & \downarrow
\end{array} \right].
\intertext{since $\nu = \lambda f$, we have}
\left[  \begin{array}{ccc}
\uparrow    &         & \uparrow     \\
\nu         & \ldots  & \nu^{(n+1)}  \\
\downarrow  &         & \downarrow
\end{array} \right]   &= \left[  \begin{array}{ccc}
\uparrow    &         & \uparrow   \\
f           & \ldots  & f^{(n+1)}  \\
\downarrow  &         & \downarrow
\end{array} \right] \Lambda ,
\end{align*}
where
\[
\Lambda   = \left[ \begin{array}{lllllll}
\lambda   & \lambda'  & \lambda''  & \lambda'''  & \ldots   
& \binom{n}{0}\lambda^{(n)}        & \binom{n+1}{0}\lambda^{(n+1)} \\
0         & \lambda   & 2\lambda'  & 3\lambda''  & \ldots   
& \binom{n}{1}\lambda^{(n-1)}      & \binom{n+1}{1}\lambda^{(n)} \\
0         & 0         & \lambda    & 3\lambda'   & \ldots   
& \binom{n}{2}\lambda^{(n-2)}      & \binom{n+1}{2}\lambda^{(n-1)} \\
0         & 0         & 0          & \lambda     & \ldots
& \binom{n}{3}\lambda^{(n-3)}      & \binom{n+1}{3}\lambda^{(n-2)} \\
\vdots    & \vdots    & \vdots     & \vdots      & \ddots
& \vdots                           & \vdots \\
0         & 0         & 0          & 0           & \ldots
& \binom{n}{n}\lambda              & \binom{n+1}{n}\lambda' \\
0         & 0         & 0          & 0           & \ldots
& 0                                & \binom{n+1}{n+1}\lambda
\end{array} \right]
\]
is an $(n+2)\times (n+2)$ matrix. The fact that the vector valued functions $ f, \ldots ,f^{(n)}$ form a linearly independent set tells us that the Wronskian is different from zero. Then we know that $ f, \ldots ,f^{(n)} $ are solutions of an $(n+1)$st-order differential equation of the form
\[
h_{n+1}f^{(n+1)}(z)+\ldots +h_{0}f(z)=0.
\]
This shows that there is a choice of functions $ g_0, \ldots ,g_n$ such that
\begin{align}
\label{eq:glambda}
f^{(n+1)}  &= g_0 f + \ldots + g_n f^{(n)}.\nonumber
\intertext{Thus writing $f^{(n+1)}$ in terms of the lower derivatives we have}
\left[  \begin{array}{ccc}
\uparrow    &         & \uparrow     \\
\nu         & \ldots  & \nu^{(n+1)}  \\
\downarrow  &         & \downarrow
\end{array} \right]   &= \left[  \begin{array}{ccc}
\uparrow    &         & \uparrow   \\
f           & \ldots  & f^{(n)}  \\
\downarrow  &         & \downarrow
\end{array} \right] \mbox{G} \;\Lambda
\intertext{where}
\mbox{G}         & = \left[ \begin{array}{ccccc}
1      & 0       & \ldots   & 0        & g_0 \\
0      & 1       & \ldots   & 0        & g_1 \\
\vdots & \vdots  & \ddots   & \vdots   & \vdots\\
0      & 0       & \ldots   & 1        & g_n 
\end{array} \right] \nonumber
\end{align}
is an $(n+1)\times (n+2)$ matrix.The product of G and $\Lambda$ is also an $ (n+1)\times (n+2)$ matrix which we call H.

\begin{equation}
H = G \Lambda = \left[ \begin{array}{lllllll}
\lambda   & \lambda'  & \lambda''  & \lambda'''  & \ldots   
& \binom{n}{0}\lambda^{(n)}        & \binom{n+1}{0}\lambda^{(n+1)}+\lambda g_0 \\
0         & \lambda   & 2\lambda'  & 3\lambda''  & \ldots   
& \binom{n}{1}\lambda^{(n-1)}      & \binom{n+1}{1}\lambda^{(n)}+\lambda g_1 \\
0         & 0         & \lambda    & 3\lambda'   & \ldots   
& \binom{n}{2}\lambda^{(n-2)}      & \binom{n+1}{2}\lambda^{(n-1)}+\lambda g_2 \\
0         & 0         & 0          & \lambda     & \ldots
& \binom{n}{3}\lambda^{(n-3)}      & \binom{n+1}{3}\lambda^{(n-2)} +\lambda g_3\\
\vdots    & \vdots    & \vdots     & \vdots      & \ddots
& \vdots                           & \vdots \\
0         & 0         & 0          & 0           & \ldots
& \binom{n}{n}\lambda              & \binom{n+1}{n}\lambda' +\lambda g_n
\end{array} \right].
\end{equation}
Now we have shown that

\begin{equation*}
\left[ \begin{array}{ccc}
\uparrow    &         & \uparrow     \\
\nu         & \ldots  & \nu^{(n+1)}  \\
\downarrow  &         & \downarrow
\end{array} \right]   = \left[  \begin{array}{ccc}
\uparrow    &         & \uparrow   \\
f           & \ldots  & f^{(n)}  \\
\downarrow  &         & \downarrow
\end{array} \right] \mbox{H}\, ,
\end{equation*}
from which we can calculate the formulas for the $\kappa$'s. All we need to obtain is the signed matrix of the $\nu$'s with the $(j+1)$st column deleted, for $\kappa_j$. For this reason we construct the $(n+1)\times (n+1)$ matrix $\mbox{H}_j$ obtained by deleting the $(j+1)$st column from H, then

\begin{align}
\label{eq:kappaj}
\kappa_j    &= (-1)^{n-j} \, det \, 
\left[ \begin{array}{cccccc}
\uparrow    &         &\uparrow   & \uparrow      &        & \uparrow     \\
\nu         & \ldots  &\nu^{(j)}  & \nu^{(j+2)}   &\ldots  & \nu^{(n+1)}  \\
\downarrow  &         &\downarrow & \downarrow    &        &\downarrow
\end{array} \right]  \nonumber \\
            &= (-1)^{n-j} \, det \, 
\left( \left[  \begin{array}{ccc}
\uparrow    &         & \uparrow   \\
f           & \ldots  & f^{(n)}  \\
\downarrow  &         & \downarrow
\end{array} \right] \mbox{H}_j \right) \nonumber \\
            &= (-1)^{n-j} \, det \, \mbox{H}_j \, det \, 
\left[  \begin{array}{ccc}
\uparrow    &         & \uparrow   \\
f           & \ldots  & f^{(n)}  \\
\downarrow  &         & \downarrow
\end{array} \right].
\end{align}

\subsection{The $\kappa$'s for Curves in Low Dimensions}
\subsubsection{ $n = 1$ }
Curves in the one--dimensional projective space $\mathbb{CP}$, and liftings to the two--dimensional complex space $\mathbb{C}^{2} \setminus \{ 0 \}$ can be expressed by $\Phi = x(z)$, and $f=(1,x)$. Then
\begin{equation}
\label{eq:ff'f''}
\left\{ \begin{array}{l}
f= (1,x),\\
 \\
f'=(0,x'),\\
 \\
f''=(0,x'')=\left( \frac{x''}{x'} \right) (0,x') = \left( \frac{x''}{x'} \right) f'. \\
\end{array} \right.
\end{equation}
Using these in the equation (\ref{eq:lambdan+1}), we get

\begin{align*}
\lambda ^{2} \; det \left[ \begin{array}{cc}
1  & 0 \\
x  & x' \\
\end{array} \right] = 1 \quad & \Longrightarrow  \quad \left\{ \begin{array}{l}
\lambda = (x')^{-1/2} \\
 \\
\lambda' = \frac{-1}{2}{(x')}^{-3/2}x'' \\
 \\
\lambda'' = \frac{3}{4} {(x')}^{-5/2}{x''}^{2} - \frac{1}{2}{(x')}^{-3/2}x'''  \\
\end{array} \right. 
\intertext{and using the fact that we can write $f^{(n+1)}  = g_0 f + \ldots + g_n f^{(n)}$, we see from the equations (\ref{eq:ff'f''}) that $g_0$ must be equal to zero and $g_1$ equal to $(x''/x')$ which furthermore is equal to $ (-2 \lambda' / \lambda)$, since $f'' = 0 f + (x''/x') f'$. Then from (\ref{eq:glambda}) we get}
\left[ \begin{array}{ccc}
\uparrow   & \uparrow   & \uparrow   \\
\nu        & \nu'       & \nu''      \\
\downarrow & \downarrow & \downarrow 
\end{array} \right]     
&= \left[ \begin{array}{cc}
1  & 0 \\
x  & x' 
\end{array} \right] 
\left[ \begin{array}{ccc}
1  & 0  & g_0  \\
0  & 1  & g_1
\end{array} \right] 
\left[ \begin{array}{ccc}
\lambda   & \lambda'   & \lambda''  \\
0         & \lambda    & 2\lambda'  \\
0         & 0          & \lambda
\end{array} \right],
\end{align*}
from which we obtain H and $\mbox{H}_{0}$.
\[
\mbox{H}   = \left[ \begin{array}{ccc}
\lambda    & \lambda'  & (\lambda'' + \lambda g_0)  \\
0          & \lambda   & ( 2 \lambda' + \lambda g_1 )
\end{array} \right]
\]
\[
\mbox{H}_0 = \left[ \begin{array}{cc}
\lambda'   & (\lambda'' + \lambda g_0) \\
\lambda    & (2\lambda' + \lambda g_1)
\end{array} \right]
\]
Then from (\ref{eq:kappaj}) we get
\begin{align}
\label{eq:kappazero} 
\kappa_{0} & = (-1)^{1}\, det \, \mbox{H}_0 \, det \left[ \begin{array}{cc}
1  &  0 \\
x  &  x' 
\end{array} \right]  \nonumber \\
 \nonumber \\
& \quad = -( 2{\lambda'}^2 + \lambda \lambda' g_1 -\lambda \lambda'' + \lambda^{2} g_0)(x') = -x' \left( 2{\lambda'}^{2} + \lambda \lambda' \frac{-2\lambda'}{\lambda} -\lambda \lambda'' \right)  \nonumber \\ 
 \nonumber  \\
& \quad \quad  = -x' \lambda \lambda'' \, = \, \frac{\lambda''}{\lambda}= \left(\frac{3}{4}{(x')}^{-5/2}{x''}^{2} - \frac{1}{2}{(x')}^{-3/2}x''' \right) x^{1/2}  \nonumber  \\
 \nonumber  \\
& \quad \quad \quad = \frac{3}{4} \left( \frac{x''}{x'} \right)^{2} - \frac{1}{2}\left( \frac{x'''}{x'} \right) \nonumber \\
 \nonumber \\
& \quad \quad \quad \quad = - \frac{1}{2}  S  x(z).
\end{align}
Thus we have shown that the Schwarzian curvature of a curve in the one--dimensional projective space is a constant multiple of the Schwarzian derivative.

\begin{example}
We have calculated the Schwarzian derivative for the one--dimensional complex function $e^{z}$ in example \ref{ex:ez}, we obtained
\[ S ( e^z ) = \frac{-1}{2}. \]
Then the Schwarzian curvature of $e^z$ is nothing but the Schwarzian derivative multiplied by $-1/2$. Thus the Schwarzian curvature of the function $e^{z}$ is constant. $\kappa_0 = 1/4$. 
\end{example}

\subsubsection{ $ n=2 $ }
When $n$ is equal to two, we have $ \kappa_0 $ and $ \kappa_1 $ to compute. In this case we have a curve in the two--dimensional projective space $\mathbb{CP}^{2}$, and liftings to the three--dimensional complex space $\mathbb{C}^3 \setminus \{ 0 \}$. $\Phi = (x(z),y(z))$, and $f=(1,x(z),y(z))$. Then $f'=(0,x'(z),y'(z))$, $f''=(0,x''(z),y''(z))$, $f'''=(0,x'''(z),y'''(z))$. H is then the $(3 \times 4 )$ matrix
\begin{align*}
\mbox{H}  &= \left[ \begin{array}{llll}
\lambda   & \lambda'  & \lambda''  & \lambda''' +\lambda g_0 \\
0         & \lambda   & 2\lambda'  & 3\lambda'' +\lambda g_1 \\
0         & 0         & \lambda    & 3\lambda'  +\lambda g_2
\end{array} \right].
\intertext{The $g$'s can easily be calculated by Gaussian elimination. From}
\left[ \begin{array}{c}
0    \\
x''' \\
y''' 
\end{array} \right] &= g_0 \left[ \begin{array}{c}
1    \\
x'' \\
y'' 
\end{array} \right] + g_1 \left[ \begin{array}{c}
0    \\
x' \\
y' 
\end{array} \right] + g_2 \left[ \begin{array}{c}
0    \\
x'' \\
y''' 
\end{array} \right]
\end{align*}
we directly see that $g_0$ must be equal to zero, and

\begin{align*}
\left[ \begin{array}{cc}
x' & x''\\
y' & y'' \end{array}
\left| \begin{array}{c}
x''' \\
y''' \end{array} \right. \right] 
& \Longrightarrow  \left[ \begin{array}{cc}
1  & \frac{x''}{x'}\\
0  & y'' -\frac{y'x''}{x'} \end{array}
\left| \begin{array}{c}
\frac{x'''}{x'} \\
y''' - \frac{y'x''}{x'} \end{array} \right. \right]  \quad  \Longrightarrow \\
 \\
\left[ \begin{array}{cc}
1  & \frac{x''}{x'}\\
 \\
0  & 1 \end{array}
\left| \begin{array}{c}
\frac{x'''}{x'} \\
 \\
\frac{y''' x' - y' x'''}{x'}\frac{x'}{y'' x' - y' x''} \end{array} \right. \right] 
& \Longrightarrow
\left[ \begin{array}{cc}
1  & 0\\
 \\
0  & 1 \end{array}
\left| \begin{array}{c}
\frac{x''' y'' - x'' y'''}{x' y''- x'' y' } \\
 \\
\frac{y''' x' - y' x'''}{y'' x' - y' x''} \end{array} \right. \right] 
\end{align*}
gives the expressions for the other $g$'s.

\begin{align*}
g_1 &= \frac{x''' y'' - x'' y'''}{x' y''- x'' y' } \\
 \\
g_2 &=\frac{x' y''' - x''' y'}{x' y'' - x'' y' } \\
\intertext{The $\lambda$'s we calculate using equation (\ref{eq:lambdan+1}) except that for simplicity we use the notation $\sigma_{ij}$ for the expression $x^{(i)}y^{(j)}-x^{(j)}y^{(i)}$.}
 &\left\{ \begin{array}{l}
\lambda = \frac{1}{(x'y''-x''y')^{1/3}} = \frac{1}{(\sigma_{12} )^{1/3}} \\
 \\
\lambda' = -\frac{1}{3}\frac{\sigma_{13} }{ (\sigma_{12})^{4/3} } \\
 \\
\lambda'' = \frac{4}{9}\frac{(\sigma_{13})^{2} }{ (\sigma_{12})^{7/3} } - \frac{1}{3}\frac{\sigma_{23}+\sigma_{14} }{ (\sigma_{12})^{4/3} } \\ 
 \\
\lambda''' = \frac{4}{3}\frac{\sigma_{13} (\sigma_{23} + \sigma_{14})}{ (\sigma_{12})^{7/3} } - \frac{1}{3}\frac{\sigma_{15} + 2 \sigma_{24}}{(\sigma_{12})^{4/3} }- \frac{28}{27}\frac{(\sigma_{13})^{3}}{(\sigma_{12})^{10/3}}
\end{array} \right.
\end{align*}
One should also bare in mind that these formulas give arise to other useful expressions as long as we allow ``manipulating'' them, for example $g_1 = (-\sigma_{23}/ \sigma_{12})$, $g_2 = (\sigma_{13}/ \sigma_{12})$ or $g_2 = (-3\lambda' / \lambda)$  etc. \\

The $\kappa$'s we can calculate from (\ref{eq:kappaj}).
\begin{align*}
\kappa_1  &= (-1)^{1} \, det \left[ \begin{array}{ccc}
1   & x   & y \\
0   & x'  & y' \\
0   & x'' & y'' \\
\end{array} \right] \, det \left[ \begin{array}{ccc}
\lambda   & \lambda''   & \lambda''' +\lambda g_0 \\
0         & 2\lambda'   & 3\lambda'' +\lambda g_1 \\
0         & \lambda     & 3\lambda'  +\lambda g_2
\end{array} \right]  \\
 \\
 &= -(x'y''-x''y')( 6 \lambda {\lambda'}^{2} + 2 \lambda^{2} \lambda' g_2 - 3 \lambda^{2} \lambda'' - \lambda^{3} g_1 ),
\intertext{since $(x'y''-x''y') = \lambda^{-3}$ and $g_2 = (-3\lambda' / \lambda)$, we get}
\kappa_{1} &= - \frac{6 {\lambda'}^{2}}{\lambda^{2}} - \frac{2 \lambda' g_2}{\lambda} +\frac{3\lambda''}{\lambda} + g_1 = \frac{3 \lambda''}{\lambda} + g_1 \\
\intertext{And}
\kappa_0  &= (-1)^{0} \, det \left[ \begin{array}{ccc}
1   & x   & y \\
0   & x'  & y' \\
0   & x'' & y'' \\
\end{array} \right] \, det \left[ \begin{array}{ccc}
\lambda'   & \lambda''   & \lambda''' +\lambda g_0 \\
\lambda    & 2\lambda'   & 3\lambda'' +\lambda g_1 \\
0          & \lambda     & 3\lambda'  +\lambda g_2
\end{array} \right]  \\
 \\
 &= (x'y''-x''y')( 6{\lambda'}^{3} + 2 {\lambda'}^{2}\lambda g_2 +\lambda^{2}\lambda''' + \lambda^{3} g_0 - 6\lambda \lambda' \lambda'' - \lambda^{2} \lambda' g_1 - \lambda^{2}\lambda'' g_2 ) \\
 \\
 &=(x'y''-x''y')( 6{\lambda'}^{3} + 2 {\lambda'}^{2}\lambda g_2 - 3\lambda \lambda' \lambda'' - \lambda^{2}\lambda' g_1 ) \\
 \\
 &+ \lambda^{-3}( \lambda^{2}\lambda''' + \lambda^{3} g_0 - 3\lambda \lambda' \lambda'' - \lambda^{2}\lambda'' g_2 )
\intertext{Here we see that the elements in the first part are the same as $\kappa_1$ multiplied with $-\lambda' / \lambda$. Thus}
\kappa_0 &= -\frac{\lambda'}{\lambda}\kappa_1 + \frac{\lambda'''}{\lambda}+\lambda^{-3} ( -3\lambda \lambda' \lambda'' - \lambda \lambda'' g_2 ) = -\frac{\lambda'}{\lambda} \kappa_1 + \frac{\lambda'''}{\lambda}
\end{align*}
Using the $\sigma$'s we obtain the final expressions for the $\kappa$'s in terms of the derivatives of $x(z)$ and $y(z)$. First we express the fractions $\lambda' / \lambda$ and ...

\begin{align}
& \left\{ \begin{array}{l}
\frac{\lambda'}{\lambda} = \frac{-\sigma_{13}}{3\sigma_{12}} \\
 \\
\frac{\lambda''}{\lambda} = \frac{4}{9} \left( \frac{\sigma_{13}}{\sigma_{12}} \right)^{2} - \frac{1}{3} \frac{\sigma_{23} +\sigma_{14}}{\sigma_{12}} \\
 \\
\frac{\lambda'''}{\lambda} = \frac{4}{3}\frac{\sigma_{13} (\sigma_{23} + \sigma_{14})}{(\sigma_{12})^{2}} - \frac{1}{3} \frac{\sigma_{15} + 2 \sigma_{24}}{\sigma_{12}} - \frac{28}{27} \left( \frac{\sigma_{13}}{\sigma_{12}} \right)^{3} 
\end{array} \right. \nonumber \\ 
 \nonumber \\
 \nonumber \\
\label{eq:kappa0}
\kappa_0 &= \frac{4}{9} \left( \frac{\sigma_{13}}{\sigma_{12}} \right)^{3} - \frac{ \sigma_{13}\sigma_{14}}{3(\sigma_{12})^{2}} - \frac{2}{3} \frac{\sigma_{13}\sigma_{23}}{(\sigma_{12})^{2}} + \frac{4}{3}\frac{\sigma_{13} \sigma_{23}}{(\sigma_{12})^{2}} \nonumber \\
 \nonumber \\
         &+ \frac{4}{3}\frac{\sigma_{13} \sigma_{14}}{(\sigma_{12})^{2}} - \frac{1}{3}\frac{\sigma_{15} + 2\sigma_{24}}{\sigma_{12}} - \frac{28}{27}\left( \frac{\sigma_{13}}{\sigma_{12}} \right)^{3} \nonumber \\
 \nonumber \\
         &= -\frac{16}{27} \left( \frac{\sigma_{13}}{\sigma_{12}} \right)^{3} + \frac{1}{3}\frac{\sigma_{13} ( 3\sigma_{14} + 2 \sigma_{23} )}{(\sigma_{12})^{2}} - \frac{1}{3}\frac{\sigma_{15} + 2\sigma_{24}}{\sigma_{12}}. \\
 \nonumber \\
 \nonumber \\
\label{eq:kappa1}
\kappa_1 &= \frac{4}{3} \left( \frac{\sigma_{13}}{\sigma_{12}} \right)^{2} - \frac{\sigma_{23} +\sigma_{14}}{\sigma_{12}} + \frac{\sigma_{23}}{\sigma_{12}} \nonumber \\
 \nonumber \\
         &= \frac{4}{3} \left( \frac{\sigma_{13}}{\sigma_{12}} \right)^{2} - \frac{\sigma_{14} + 2\sigma_{23}}{\sigma_{12}} .
\end{align}
From the equations (\ref{eq:kappa0}) and (\ref{eq:kappa1}) above we also see that one important condition for the existence of $\kappa_0$ and $\kappa_1$, is that $\sigma_{12}$ must be different from zero.

\begin{example}
As an example for a curve in the two--dimensional projective space, and its lifting to $\mathbb{C}^{3} \setminus \{ 0 \}$, we choose $\Phi = (\cos z , \sin z )$. Which when lifted to the three--dimensional complex space, seems to be the complex circle $f = (1, \cos z , \sin z )$. Then we calculate

\[
\begin{array}{lllll}
\left\{ \begin{array}{l}
x(z)=\cos z  \\
x'(z)=-\sin z  \\
x''(z)=-\cos z   \\
x'''(z)= \sin z    \\
x^{(4)}(z)= \cos z   \\
x^{(5)}(z) =-\sin z    
\end{array} \right. & & \left\{ \begin{array}{l}
y(z)=\sin z   \\
y'(z)=\cos z    \\
y''(z)=-\sin z    \\
y'''(z)= -\cos z    \\
y^{(4)}(z)=  \sin z   \\
y^{(5)}(z) =  \cos z    
\end{array} \right. & & \left\{ \begin{array}{l}
\sigma_{12} = 1  \\
\sigma_{13} = 0  \\
\sigma_{14} = -1 \\
\sigma_{15} = 0  \\
\sigma_{23} = 1  \\
\sigma_{24} = 0. 
\end{array} \right. 
\end{array}
\]
Hence
\begin{align*}
\kappa_{0} &= 0,\\
\kappa_{1} &= - 1 .
\end{align*}
\end{example}

\newpage
\subsubsection{ $n = 3$ }
For $n=3$, we have $\kappa_{i} \, , \, i=0,1,2$ to compute. Looking back to the cases $n=1$ and $n=2$, we see that the formulas for the $\kappa$'s differ a lot. Naturally the Schwarzian curvatures for curves in the three--dimensional complex projective space will require more calculations then those for curves in $\mathbb{CP}$ and $\mathbb{CP}^{2}$. This is a good reason to let a computer take care of the calculations. For this purpose we use Mathematica, which is good at symbolic calculations. In fact, the formulas will be so big that even writing them will be difficult.\\

We start with expressing curves in $\mathbb{CP}^{3}$ and their liftings to the four--dimensional complex space by $\Phi =(w(z),x(z),y(z))$ , $f=(1,w,x,y)$. Then $ f' = (0,w',x',y')$ , $f''=(0,w'',x'',y'')$ , $f'''=(0,w''',x''',y''')$ and $f^{(4)}=(0,w^{(4)},x^{(4)},y^{(4)})$. The $g$'s are obtained from
\begin{equation}
\left[ \begin{array}{c}
0 \\ w^{(4)} \\ x^{(4)} \\ y^{(4)}
\end{array} \right]
= g_{0}\left[ \begin{array}{c}
1 \\ w \\ x \\ y
\end{array} \right]
+ g_{1}\left[ \begin{array}{c}
0 \\ w' \\ x' \\ y'
\end{array} \right]
+ g_{2}\left[ \begin{array}{c}
0 \\ w'' \\ x'' \\ y''
\end{array} \right]
+ g_{3} \left[ \begin{array}{c}
0 \\ w''' \\ x''' \\ y'''
\end{array} \right]
\end{equation}
We let Mathematica do the Gaussian elimination, except that when writing down, we use the notation $\sigma_{ijk}= w^{(i)} x^{(j)}  y^{(k)} - w^{(i)} x^{(k)} y^{(j)}$. Then

\begin{equation}
\label{eq:tredimg}
\left\{ \begin{array}{l}
g_{0}= 0, \\
 \\
g_{1}=\frac{\sigma_{423}+\sigma_{342}+\sigma_{234}}{\sigma_{123}+\sigma_{231}+\sigma_{312}}, \\
 \\
g_{2}=-\frac{\sigma_{413}+\sigma_{341}+\sigma_{134}}{\sigma_{123}+\sigma_{231}+\sigma_{312}}, \\
 \\
g_{3}=\frac{\sigma_{412}+\sigma_{241}+\sigma_{124}}{\sigma_{123}+\sigma_{231}+\sigma_{312}}. 
\end{array} \right.
\end{equation}
We also know that H is the $ 4 \times 5 $ matrix
\[
\mbox{H}  = \left[ \begin{array}{lllll}
\lambda   & \lambda'  & \lambda''  & \lambda'''  & \lambda^{(4)} +\lambda g_0 \\
0         & \lambda   & 2\lambda'  & 3\lambda''  & 4\lambda'''+\lambda g_1 \\
0         & 0         & \lambda    & 3\lambda'   & 6\lambda'' +\lambda g_2 \\
0         & 0         & 0          & \lambda     & 4\lambda' +\lambda g_3
\end{array} \right],
\]
which gives us $\mbox{H}_{j} \, , \, j=0,1,2$ by deleting the $(i+1)$st column of H.
\[
\begin{array}{l}
\mbox{H}_{0} = \left[ \begin{array}{llll}
 \lambda'  & \lambda''  & \lambda'''  & \lambda^{(4)} +\lambda g_0 \\
 \lambda   & 2\lambda'  & 3\lambda''  & 4\lambda'''+\lambda g_1 \\
 0         & \lambda    & 3\lambda'   & 6\lambda'' +\lambda g_2 \\
 0         & 0          & \lambda     & 4\lambda' +\lambda g_3
\end{array} \right], \\
 \\
\mbox{H}_{1} = \left[ \begin{array}{llll}
\lambda   & \lambda''  & \lambda'''  & \lambda^{(4)} +\lambda g_0 \\
0         & 2\lambda'  & 3\lambda''  & 4\lambda'''+\lambda g_1 \\
0         & \lambda    & 3\lambda'   & 6\lambda'' +\lambda g_2 \\
0         & 0          & \lambda     & 4\lambda' +\lambda g_3
\end{array} \right], \\
 \\
\mbox{H}_{2} = \left[ \begin{array}{lllll}
\lambda   & \lambda'  & \lambda'''  & \lambda^{(4)} +\lambda g_0 \\
0         & \lambda   & 3\lambda''  & 4\lambda'''+\lambda g_1 \\
0         & 0         & 3\lambda'   & 6\lambda'' +\lambda g_2 \\
0         & 0         & \lambda     & 4\lambda' +\lambda g_3
\end{array} \right].
\end{array} 
\]
The $\lambda$'s we obtain by applying equation (\ref{eq:lambdan+1}).
\begin{align}
\label{eq:tredimlambdawxy}
1  & = \lambda^{(4)} \, det \, \left[ \begin{array}{cccc}
1  & 0   & 0    & 0       \\
w  & w'  & w''  & w'''    \\
x  & x'  & x''  & x'''    \\
y  & y'  & y''  & y'''  
\end{array} \right] \quad  \quad  \Longrightarrow \nonumber \\
 \nonumber \\
 \nonumber \\
\lambda  &= \frac{1}{(w'x''y'''-w'x'''y''+w''x'''y'-w''x'y'''+w'''x'y''-w'''x''y')^{1/4}}. \\
 \nonumber 
\end{align}
Unfortunately, not even Mathematica could derive (\ref{eq:tredimlambdawxy}) four times, so we will not be able to give the complete formulas for the $\kappa$'s. At least it wouldn't give us a better understanding about the behaviour of the $\kappa$'s even if we could do that. Instead we use the notation, that we introduced for the $g$'s in (\ref{eq:tredimg}), for giving the formulas for the $\lambda$'s. It is also useful to use some properties of the $\sigma$'s, for example
\[
\left\{ \begin{array}{l}
\sigma_{ijk}=0  \mbox{    for    }  j=k ,\\
 \\
\sigma_{ijk}+\sigma_{ikj}=0.
\end{array} \right.
\]
Using these and some other substitutions, we get the following expressions:

\begin{align}
\label{eq:tredimlambda}
\lambda     & = \frac{1}{(\sigma_{123}+\sigma_{231}+\sigma_{312})^{1/4}} \\
 \nonumber \\
\label{eq:tredimlambda'}
\lambda'    & = -\frac{(\sigma_{124}+\sigma_{241}+\sigma_{412})}{4(\sigma_{123}+\sigma_{231}+\sigma_{312})^{5/4}} = -\frac{\lambda g_3 }{4} \\
 \nonumber \\
\label{eq:tredimlambda''}
\lambda''   & = \frac{15}{16} \left( \frac{\sigma_{124}+\sigma_{241}+\sigma_{412}}{\sigma_{123}+\sigma_{231}+\sigma_{312}} \right)^{2} -\frac{\sigma_{134}+\sigma_{341}+\sigma_{413}}{4 (\sigma_{123}+\sigma_{231}+\sigma_{312})^{5/4}} \nonumber \\
 \nonumber \\
 & -\frac{\sigma_{125}+\sigma_{251}+\sigma_{512}}{4 (\sigma_{123}+\sigma_{231}+\sigma_{312})^{5/4}} \; = \frac{\lambda g_2}{4} - \frac{5 {g_3}^{2}}{6} + \frac{\lambda}{4} \left( \frac{\sigma_{125}+\sigma_{251}+\sigma_{512}}{\sigma_{123}+\sigma_{231}+\sigma_{312}} \right) \\
 \nonumber \\
\label{eq:tredimlambda'''}
\lambda'''  & = \frac{ 15 g_2 g_3 }{16} - \frac{\lambda g_1}{4} - \frac{ 10 {g_3}^{3}}{16}-\frac{\lambda}{2} \left( \frac{ \sigma_{135} + \sigma_{351} + \sigma_{513} }{\sigma_{123} + \sigma_{231} + \sigma_{312} } \right) \nonumber \\
 \nonumber \\
 & +\frac{15 g_{3}}{16} \left( \frac{\sigma_{125} + \sigma_{251} + \sigma_{512}}{ \sigma_{123} + \sigma_{231} + \sigma_{312} } \right) - \frac{\lambda}{4} \left( \frac{\sigma_{126} + \sigma_{261} + \sigma_{612}}{ \sigma_{123} + \sigma_{231} + \sigma_{312} } \right) \\
\nonumber \\
\label{eq:tredimlambda''''}
 \lambda^{(4)} & = \; \ldots \\
 \nonumber
\end{align}
As we see in the formulas (\ref{eq:tredimlambda})--(\ref{eq:tredimlambda''''}), inserting the precise formulas in those of the $\kappa$'s is not very useful. For that reason, we just write the formulas for the $\kappa$'s in terms of $\lambda$, its derivatives and the $g$'s. Equation (\ref{eq:kappaj}) gives

\begin{align}
\label{eq:tredimk0}
\kappa_{0}      &= - det \, \mbox{H}_{0} \, det \, \left[ \begin{array}{cccc}
1  & 0   & 0    & 0       \\
w  & w'  & w''  & w'''    \\
x  & x'  & x''  & x'''    \\
y  & y'  & y''  & y'''  
\end{array} \right] \nonumber \\
  \nonumber \\
& \quad = - (\sigma_{123}+\sigma_{231}+\sigma_{312})(\lambda^{2}{\lambda'}^{2}g_{1}+6\lambda {\lambda'}^{3}g_{3} - 12{\lambda'}^{3}\lambda'' + 2\lambda^{2}{\lambda'}^{2}g_{2} \nonumber\\
  \nonumber\\
& \quad \quad - 3\lambda^{2} \lambda' \lambda'' g_{3}- 24\lambda {\lambda'}^{2} \lambda'' -3\lambda {\lambda'}^{2} \lambda'' g_{3}+24{\lambda'}^{4} + \lambda^{3}\lambda'' g_{2} \nonumber \\
 \nonumber \\
& \quad \quad \quad + 6\lambda^{2} {\lambda''}^{2} +\lambda^{2} \lambda' \lambda''' g_{3} + 8\lambda {\lambda'}^{2} \lambda''' - \lambda^{2} \lambda' \lambda^{(4)}). \\
 \nonumber \\
 \nonumber \\
\label{eq:tredimk1}
\kappa_{1}      &= det \, \mbox{H}_{1} \, det \, \left[ \begin{array}{cccc}
1  & 0   & 0    & 0       \\
w  & w'  & w''  & w'''    \\
x  & x'  & x''  & x'''    \\
y  & y'  & y''  & y'''  
\end{array} \right] \nonumber \\
  \nonumber \\
& \quad = (\sigma_{123}+\sigma_{231}+\sigma_{312})(\lambda^{4} g_{1} - 2\lambda^{3} \lambda' g_{2} + 6\lambda^{2} {\lambda'}^{2}g_{3} \nonumber \\
 \nonumber \\
& \quad \quad + 24 \lambda {\lambda'}^{3} - 3\lambda^{3} {\lambda''} g_{3} - 24 \lambda^{2} \lambda' \lambda'' + 4\lambda^{3} {\lambda'''}). \\
\nonumber \\
\nonumber \\
\label{eq:tredimk2}
\kappa_{2}      &= - det \, \mbox{H}_{2} \, det \, \left[ \begin{array}{cccc}
1  & 0   & 0    & 0       \\
w  & w'  & w''  & w'''    \\
x  & x'  & x''  & x'''    \\
y  & y'  & y''  & y'''  
\end{array} \right] \nonumber \\
 \nonumber \\
& \quad = - (\sigma_{123}+\sigma_{231}+\sigma_{312})(3\lambda^{3} \lambda' g_{3} + 12 \lambda^{2} {\lambda'}^{2} - 6\lambda^{3} \lambda'' - \lambda^{4} g_{2}). \\
 \nonumber 
\end{align}
Perhaps these expressions are slightly confusing because of the lack of knowledge about their direct dependence of $w(z)$, $x(z)$ and $y(z)$. One easy way to show these relations is to calculate the $\kappa$'s for a known curve. We construct an example.

\begin{example}
In this example, we try to choose the curve $\Phi$ such that the $\kappa$'s will be relatively easy to compute. We choose $\Phi = (z^{2} / 2 , \cos z, \sin z)$. Then the lifting to the four--dimensional complex space $\mathbb{C}^{4} \setminus \{ 0 \}$ is $f(z)=(1,z^{2} / 2 , \cos z, \sin z)$. Now we have
\[
\begin{array}{lllll}
\left\{ \begin{array}{l}
w(z) = \frac{z^{2}}{2}  \\
w'(z) = z  \\
w''(z) = 1 \\
w'''(z) = 0  \\
w^{(4)}(z) = 0  \\
\vdots
\end{array} \right. & & \left\{ \begin{array}{l}
x(z)=\cos z  \\
x'(z)=-\sin z  \\
x''(z)=-\cos z   \\
x'''(z)= \sin z    \\
x^{(4)}(z)= \cos z   \\
\vdots    
\end{array} \right. & &  \left\{ \begin{array}{l}
y(z)=\sin z   \\
y'(z)=\cos z    \\
y''(z)=-\sin z    \\
y'''(z)= -\cos z    \\
y^{(4)}(z)=  \sin z   \\
\vdots 
\end{array} \right.
\end{array}
\]
We calculate the various $\sigma$'s. Then we calculate the derivatives of $\lambda$ from the obtained $\lambda$ and directly from (\ref{eq:tredimg}) we calculate the $g$'s.
\begin{align*}
& \left\{ \begin{array}{l}
\sigma_{123}=w' x'' y''' - w' x''' y'' = z({\cos }^{2} z+ {\sin }^{2} z) = z   \\
\sigma_{134}= z \\
\sigma_{241}=\sigma_{234}= 1 \\
\sigma_{231}=\sigma_{312}= \sigma_{341}=\sigma_{413}= \sigma_{124}=\sigma_{412}= \sigma_{342}=\sigma_{423}= 0 \\
\end{array} \right.  \\
 \\
& \left\{ \begin{array}{l}
\lambda = z^{-1/4} \\
 \\
\lambda' = -\frac{1}{4} z^{-5/4} \\
 \\
\lambda'' = \frac{5}{16} z^{-9/4} \\
 \\
\lambda''' = -\frac{45}{64} z^{-13/4} \\
 \\
\lambda^{(4)} = \frac{585}{256} z^{-17/4} 
\end{array} \right. \\
 \\
& \left\{ \begin{array}{l}
g_{0} = 0 \\
 \\
g_{2} = -1 \\
 \\
g_{1} = g_{3} = \frac{1}{z}
\end{array} \right.
\end{align*}
These can be directly inserted into the given formulas for the $\kappa$'s, (\ref{eq:tredimk0})--(\ref{eq:tredimk2}). We get:

\begin{align*}
\kappa_{0} &= z\left( -\frac{1}{16z^{4}} - \frac{2}{16z^{3}} + \frac{6}{64z^{5}} - \frac{24}{256z^{5}} + \frac{5}{16z^{3}} - \frac{45}{256z^{5}} + \frac{120}{256z^{5}} \right) \\
 \\
           &+ z\left( \frac{15}{256z^{6}} - \frac{60}{1024z^{6}} - \frac{150}{256z^{5}} - \frac{45}{256z^{6}} + \frac{360}{1024z^{6}} - \frac{585}{1024z^{6}} \right) \\
 \\
           &= \frac{7}{(4z)^{2}} - \frac{4}{(4z)^{3}} - \frac{75}{(4z)^{4}} - \frac{285}{(4z)^{5}}, \\
 \\
\kappa_{1} &= z \left( \frac{1}{z^{2}} - \frac{2}{4z^{2}} + \frac{6}{16z^{4}} - \frac{24}{64z^{4}} - \frac{15}{16z^{4}} + \frac{120}{64z^{4}} - \frac{180}{64z^{4}} \right) \\
 \\
           &= \frac{1}{2z} - \frac{15}{(2z)^{3}}, \\
 \\
\kappa_{2} &= z \left( -\frac{1}{z} + \frac{3}{4z^{3}} - \frac{12}{16z^{3}} + \frac{30}{16z^{3}} \right) \\
 \\
           &= - 1 + \frac{15}{8z^{2}}.
\end{align*}
\end{example}

\subsection{Transformation Formulas}
The Schwarzian curvatures are not invariant under change of coordinates, instead they obey a set of transformation rules. In this part we try to give the formulas for the transformed $\kappa$'s in the lower dimension cases. We let $z=z(w)$ be a change of coordinates in the disc $\mathcal{D}$.  

\begin{theorem}
The transformation formula for $\kappa_0$ in the one--dimensional projective space, $\mathbb{CP}$, is the following multiple of the Schwarzian derivative:
\begin{equation}
\label{eq:kappa0trans}
\tilde{\kappa}_0 = {z'}^{2}\kappa_0 - \frac{1}{2} S z.
\end{equation}
\end{theorem}

\begin{proof}
Recall the formula for $\kappa_0$ in equation (\ref{eq:kappazero})

\[
\kappa_0 = \frac{3}{4}\left( \frac{x''(z)}{x'(z)}\right)^2 - \frac{1}{2}\left( \frac{x'''(z)}{x'(z)}\right).
\]
Instead of $\Phi = x(z)$ that we started with, we use $\tilde{\Phi} = \tilde{x}(w) = x(z(w))$ which simply is the same curve after change of coordinates. Then we calculate
\[
\left\{ \begin{array}{l}
\tilde{x}'(w) = x'z' , \\
 \\
\tilde{x}''(w) = x''{z'}^{2} + z'' x' , \\
 \\
\tilde{x}'''(w) = x'''{z'}^{3} + 3x''z'z'' + x'z''' . 
\end{array} \right.
\]

\begin{align*}
\intertext{and}
\tilde{\kappa}_0 &= - \frac{1}{2} S  \tilde{x}(w) = \frac{3}{4}\left( \frac{\tilde{x}''(z)}{\tilde{x}'(z)}\right)^{2} - \frac{1}{2}\left( \frac{\tilde{x}'''(z)}{\tilde{x}'(z)}\right) \\
 \\
&= \frac{3}{4}\frac{(x'' {z'}^{2} +z'' x' )^{2} }{ {x'}^{2} {z'}^{2} }-\frac{2}{4}\frac{(x''' {z'}^{3} + 3 x'' z' z'' +  x' z''' ) x' z' }{ {x'}^{2} {z'}^{2} } \\
 \\
&= \frac{3 {x''}^{2} {z'}^{4} - 2 x' x''' {z'}^{4} }{ 4 {x'}^{2} {z'}^{2} } +\frac{ 6 x' x'' {z'}^{2} z'' - 6 x' x'' {z'}^{2} z''}{4 {x'}^{2} {z'}^{2} }+\frac{ 3 {z''}^{2} {x'}^{2} - 2 {x'}^{2} z' z''' }{ 4 {x'}^{2} {z'}^{2} } \\
 \\
&= {z'}^{2} \kappa_0 - \frac{1}{2} S z. \quad \blacksquare
\end{align*}
\end{proof}

In the same way we calculate $\tilde{\kappa}_0$ and $\tilde{\kappa}_1$ for curves in $\mathbb{CP}^2$. We use the formulas obtained in equations (\ref{eq:kappa0}) and (\ref{eq:kappa1}) and since the coordinates are changed, we replace the $\sigma$'s by $\tilde{\sigma}$'s, where the components of $\tilde{\Phi}$ are derived with respect to the $w$'s ( Note that $\tilde{\Phi} = (x(z(w)), y(z(w)))$ ). Then we calculate the various $\sigma$'s for the curve after change of coordinates, denoted by $\tilde{\sigma}$'s.
\begin{align*}
\tilde{\sigma}_{12} & = {z'}^{3} (\sigma_{12}), \\
 \\
\tilde{\sigma}_{13} & = {z'}^{4} (\sigma_{13}) + 3 {z'}^{2} z'' (\sigma_{12}), \\
 \\
\tilde{\sigma}_{23} & = {z'}^{5} (\sigma_{23}) + {z'}^{3} z'' (\sigma_{13}) - {z'}^{2} z''' (\sigma_{12}) + 3 z' {z''}^{2} (\sigma_{12}), \\ 
 \\
\tilde{\sigma}_{14} & = {z'}^{5} (\sigma_{14}) + 6 {z'}^{3} z'' (\sigma_{13}) + 4 {z'}^{2} z''' (\sigma_{12}) + 3 z' {z''}^{2} (\sigma_{12}), \\
 \\
\tilde{\sigma}_{24} & = {z'}^{6} (\sigma_{24}) + 6 {z'}^{4} z'' (\sigma_{23}) + {z'}^{4} z'' (\sigma_{14}) + 6 {z'}^{2} {z''}^{2} (\sigma_{13}) \\
                    & - {z'}^{2} z^{(4)} (\sigma_{12}) + 4 z' z'' z''' (\sigma_{12}) + 3 {z''}^{2} (\sigma_{12}), \\
 \\
\tilde{\sigma}_{15} & = {z'}^{6} (\sigma_{15}) + 10 {z'}^{4} z'' (\sigma_{14}) + 10 {z'}^{3} z''' (\sigma_{13}) + 15 {z'}^{2} {z''}^{2} (\sigma_{13}) \\
                    & + 5 {z'}^{2} {z}^{(4)} (\sigma_{12}) + 10 z' z'' z''' (\sigma_{12}). \\
\end{align*}
The $\tilde{\kappa}$'s can now be written in terms of $\tilde{\sigma}$'s and the simplified expressions are the following:

\begin{align*}
\tilde{\kappa}_{0} & = {z'}^{3}\kappa_{0} + z' z'' \kappa_{1} - S z , \\
\intertext{and}
\tilde{\kappa}_{1} & = {z'}^{2} \kappa_{1} - 2 S z .
\end{align*} 
Where $S z$ is the Schwarzian derivative of the change of coordinates $z=z(w)$.

\clearpage{\pagestyle{empty}\cleardoublepage}

\begin{appendix}
\section{Elementary Definitions and Properties}
We present some elementary notions that we use for describing behaviours and properties of functions and mappings. We will start with introducing the notion of a topological space, since the study of manifolds involves topology. The metric properties and the notion of distance are not included.

\begin{defn}
Let $X$ be an set and let $\mathcal{T} =\{ U_i : i\in I \}$ denote a certain collection of subsets of $X$. The pair $(X,\mathcal{T} )$is called a topological space if $\mathcal{T}$ satisfies the following requirements:
\begin{enumerate}
\item $\emptyset , X \in\mathcal{T}$.
\item If $J$ is any sub-collection of $I$, the family $\{ U_j : j\in J \}$ satisfies 
\[ \bigcup_{j\in J} {U_j}\subset\mathcal{T}.\]
\item If $K$ is any finite sub-collection of $I$, the family $\{ U_k : k\in K\}$ satisfies 
\[ \bigcap_{k\in K} {U_k}\subset\mathcal{T}.\] 
\end{enumerate}
The elements of $\mathcal{T}$ are called open sets.
\end{defn}

\begin{remark}
$X$ alone is often called a topological space and $\mathcal{T}$ is said to give a topology to $X$.
\end{remark}

\begin{defn}
Let $X_1$ and $X_2$ be topological spaces and let $f$ be a function from $X_1$ to $X_2$. Then $f$ is continuous at the point $a$ of its domain if for every open set $O_{b} \in X_2$ which contains $b = f(a)$ there is an open set $O_{a} \in X_1$ which contains $a$ and is such that $f(x) \in O_{b}$ for every $x \in O_{a}$, or in other words $f(O_{a}) \subset O_{b}$.
\end{defn}

\begin{defn}
A set $N_x$ such that a given point $x$ is contained in the interior of $N_x$ is called a neighborhood of the point $x$. 
\end{defn}

\begin{defn}
Let $X_1 $ and $X_2$ be topological spaces. A map $f:X_1 \longrightarrow X_2$ is a homeomorphism if it is continuous and has an inverse $f^{-1}:X_2 \longrightarrow  X_1$ which is also continuous. If there exists a homeomorphism between two topological spaces, we say that they are homeomorphic.
\end{defn}

\begin{defn}
A topological space $X$ is connected if it cannot be written as $X = X_1  \cup X_2$, where $X_1$ and $X_2$ are both open and non--empty and $ X_1 \cap X_2 \neq \emptyset $. Otherwise $X$ is called disconnected.
\end{defn}

\begin{defn}
A loop in a topological space $X$ is a continuous map $f: [0,1] \longrightarrow X$, such that $f(0)=f(1)$. If any loop in $X$ can be continuously shrunk to a point, $X$ is called simply connected.
\end{defn}

\begin{defn}
Let $X_1 $ and $X_2$ be topological spaces. A bijective map $f:X_1 \longrightarrow  X_2$ is a $C^k$--diffeomorphism if both $f:X_1 \longrightarrow  X_2$ and its inverse $f^{-1}:X_2 \longrightarrow  X_1$ are $C^k$--functions.
\end{defn}

\begin{defn}
Let $X$ be a topological space. Given any pair of distinct points $a,b \in X$ there exist disjoint open sets $O_a$ and $O_b$ in $X$ such that $a\in O_{a}$ and $b \in O_{b}$. A topological space satisfying this axiom is called a Hausdorff space.
\end{defn}

By \real{n} we will denote the set of all ordered $n$--tuples with the usual vector operation. If $x,y \in \mathbb{R}^{n}$, then $x\cdot y$ will denote the usual scalar product and $|x|$ the Euclidean norm of $x$. \\

\begin{defn}
Let $x_1 \in \mathbb{R}^n$. An affine transformation of \real{n} is a map $ F : \mathbb{R}^n \longrightarrow \mathbb{R}^n $ of the form
\begin{align*}
F(x)=Ax + x_1
\end{align*}
for all $ x \in \mathbb{R}^n $, where $A$ is a linear transformation of \real{n}.
\end{defn}

\begin{defn}
The Jacobian matrix of a function $f(x)$ will be denoted by $\mathcal{J}_{f(x)}$, and defined by
\[
\mathcal{J}_{f(x)}=\frac{\partial f(x_1 , \ldots , x_n)}{\partial (x_1 , \ldots , x_n)}.
\]
We use the notation $d_{a}f$ for the value of the Jacobi matrix at a certain point $a$.
\end{defn}

\begin{defn}
The Wronskian with respect to the functions $f_{0}(x) ,\ldots ,f_{n}(x)$ is denoted by $W(x)$ and defined by the following determinant.
\[
W(x)=det \, \left[ \begin{array}{ccc}
f_{0}(x)         & \ldots  &f_{n}(x)         \\
\vdots           &         &\vdots           \\
f_{0}^{(n)}(x)   & \ldots  &f_{n}^{(n)}(x)  
\end{array} \right].
\]
\end{defn}
One, for us, useful property of the Wronskian is that having the system of $(n+1)$st-order linear differential equation
\[
h_{n+1}f^{(n+1)}(z)+\ldots +h_{0}f(z)=0,
\]
then $f_{0}(x),\ldots ,f_{n}(x)$ is a fundamental system of solutions if and only if the Wronskian is different from zero. This is attained when $f_{0}(x),\ldots ,f_{n}(x)$ are linearly independent.

Some properties of differentiable functions, that are very useful in developing function theory and are used in various places in this thesis, are that $f+g$, $fg$ and $f/g$ are differentiable at a point $x\in [a,b]$, when $f$ and $g$ are defined on $[a,b]$ and are differentiable at $x$. These properties are valid for analyticity of functions as well.\\
Furthermore we define a $C^k$--function by a function that has continuous partial derivatives up to order $k$, and a $C^{\infty}$--function by a function that have continuous partial derivatives up to order $ \infty $, i.e. ${f_1^{(\infty )}}, \ldots , {f_n^{(\infty )}}$ are continuous.\\

\clearpage{\pagestyle{empty}\cleardoublepage}

\section{The Schwarzian Derivative}
We also describe the Schwarzian derivative that is a tool first introduced into the study of one-dimensional dynamical systems. The Schwarzian derivative plays a very important role in complex analysis and it is a valuable tool in one-dimensional dynamics.\\

\begin{defn}
The Schwarzian derivative of a function $f$ at a point $z$ is
\begin{equation}
Sf(z)=\frac{f'''(z)}{f'(z)}-\frac{3}{2}\left( \frac{f''(z)}{f'(z)} \right)^2
=\left( \frac{f''(z)}{f'(z)} \right)'-\frac{1}{2}\left( \frac{f''(z)}{f'(z)} \right)^2.
\end{equation}
\end{defn}

\begin{example}
\label{ex:ez}
Let $f(z)=e^z$, then the Schwarzian derivative of $f(z)$ can be calculated.
\begin{align*}
f'(z) &=e^z,\\
f''(z) &=e^z,\\
f'''(z) &=e^z.\\
\intertext{Thus the Schwarzian derivative is:}
Sf(z) &=-\frac{1}{2}.
\end{align*}
\end{example}

\begin{lemma}
Let $P(x)$ be a polynomial. If all of the roots $P'(x)$ are real and distinct, 
then \mbox{$SP(x)<0.$}\\
\end{lemma}

\begin{proof}
\begin{align*}
P'(x) &=\prod_{i=1}^{N}(x-a_i),\\
 \\
P''(x) &=\sum_{j=1}^{N}\frac{P'(x)}{x-a_j}=\sum_{j=1}^{N}\frac{\prod_{i=1}^{N}(x-a_i)}{x-a_j},\\
 \\
P'''(x) &=\sum_{k=1}^{N}\frac{P''(x)}{x-a_k}=\sum_{j=1}^{N}\,\,\sum_{k=1 \atop k \neq j}^{N}\frac{\prod_{i=1}^{N}(x-a_i)}{(x-a_j)(x-a_k)}.\\
\intertext{Hence we have:}
SP(x) &=\sum_{k\neq j}\left( \frac{1}{(x-a_j)(x-a_k)}\right)-\frac{3}{2}\left( \sum_{j=1}^{N}\frac{ 1}{x-a_j} \right)^2=\sum_{k\neq j}\left( \frac{1}{(x-a_j)(x-a_k)}\right)\\
 \\
&-\frac{3}{2}\left( \sum_{j=1}^{N}\frac{1} {x-a_j} \right)^2-\frac{1}{2}\left( \sum_{j=1}^{N}\frac{1}{x-a_j} \right)^2+\frac{1}{2}\left( \sum_{j=1}^{N}\frac{1 }{x-a_j} \right)^2\\
 \\
&=-\frac{1}{2}\sum_{j=1}^{N}\left(\frac{1}{x-a_j}\right)^2-\left(\sum_{j=1}^{N}\frac{1}{x-a_j}\right)^2<0.\quad \blacksquare
\end{align*}
\end{proof}
One of the most important properties of functions which have negative Schwarzian derivative is the fact that this property is preserved under composition.

\begin{lemma}
\label{sec:Scomposition}
Suppose $Sf(x)<0$, and $Sg(x)<0$, then the Schwarzian derivative of the composition $S(f\circ g)<0.$\\
\end{lemma}

\begin{proof}
\begin{align*}
(f\circ g)' &= g'(x)\cdot f'(g(x))\\
 \\
(f\circ g)'' &=g''(x)\cdot f'(g(x))+(g'(x))^2 \cdot f''(g(x))\\ 
 \\
(f\circ g)''' &= f'''(g(x))\cdot (g'(x))^3 + g'''(x)\cdot f'(g(x)) + 3f''(g(x))\cdot g''(x)\cdot g'(x)\\
\intertext{Thus the Schwarzian derivative is:}
S(f\circ g) &= \left( \frac{f'''{g'}^3 +g'''f' +3f''g''g'}{g'f'}\right)-\frac{3}{2}\left(\frac{g''f' +f''{g'}^2}{g'f'}\right)^2\\
 \\
&=\frac{2f'''f'{g'}^4 +2g'''g'{f'}^2 +6f''g''f'{g'}^2 -3{g''}^2{f'}^2 -3{f''}^2{g'}^4 -6g''f''f'{g'}^2}{2{g'}^2 {f'}^2}\\
 \\
&=\frac{2f'''f'{g'}^2 - 3{f''}^2{g'}^2}{2{f'}^2} +\frac{2g'''g' -3{g''}^2}{2{g'}^2} \\
 \\
&=Sf(x)\cdot (g'(x))^2 +Sg(x) <0.\quad \blacksquare
\end{align*}
\end{proof}
Another very important property of the Schwarzian derivative is its invariance under fractional-linear transformations, i.e. M\"{o}bius transformations. 

\begin{theorem}
The Schwarzian derivative of a function $f(z)$ is invariant under M\"{o}bius transformations $w=\frac{az+b}{cz+d}$, i.e. $S(w\circ f)(z)=Sf(z)$.\\
\end{theorem}

\begin{proof}
Comparing with $S(f\circ g)$ from lemma \ref{sec:Scomposition}, we see that $S(w\circ f)(z)\\
=Sw(z)\cdot(f'(z))^2+Sf(z).$  Moreover we have:

\begin{align*}
w'(z) &=\frac{ad-bc}{(cz+d)^2}\\
 \\
w''(z) &=\frac{-2c(ad-bc)}{(cz+d)^3}\\
 \\
w'''(z) &=\frac{6c^2(ad-bc)}{(cz+d)^4}\\
\intertext{The Schwarzian derivative of the transformation $w(z)$ is:}
Sw(z) &=\frac{w'''(z)}{w'(z)}-\frac{3}{2}\left( \frac{w''(z)}{w'(z)}\right)^2\\
 \\
&=\frac{6c^2(ad-bd)}{(cz+d)^4}\cdot \frac{(cz+d)^2}{(ad-bc)}\\
 \\
&-\frac{3}{2}\left( \frac{-2c(ad-bc)}{(cz+d)^3}\cdot \frac{(cz+d)^2}{(ad-bc)}\right)^2=0\\
\intertext{Thus $Sw(z)=0$ and  we obtain the desired formula, $S(w\circ f)(z)=Sf(z)$.\quad $\blacksquare$}
\end{align*}
\end{proof}

\end{appendix}

\clearpage{\pagestyle{empty}\cleardoublepage}

\nocite{*}
\bibliographystyle{plain}
\bibliography{references}

\end{document}